\documentclass{amsart}
\usepackage{amssymb,amsmath,amsthm,epsf,epsfig,dsfont,bbm}
\usepackage[margin=90pt]{geometry}
\usepackage{verbatim}
\usepackage{latexsym}
\usepackage{graphicx}
\usepackage[utf8]{inputenc}
\usepackage[backref=page]{hyperref}
\hypersetup{
	colorlinks   = true,
	citecolor    = blue,
	linkcolor    = purple 
}
\usepackage{upref, eucal}
\usepackage[all]{xy}

\newcommand {\nc} {\newcommand}
\newcommand {\enm} {\ensuremath}

\def \d{\delta}

\nc {\bdm} {\begin{displaymath}}
\nc {\edm} {\end{displaymath}}

\newtheorem {theorem} {\bf{Theorem}}[section]
\newtheorem {lemma}[theorem] {\bf Lemma}
\newtheorem {proposition}[theorem] {\bf Proposition}

\newtheorem {definition}[theorem]{\bf Definition}
\newtheorem {corollary}[theorem] {\bf Corollary}
\numberwithin {equation}{section}

\newcommand\BB{\mathbb{B}}\newcommand\FF{\mathbb{F}}\newcommand\HH{{\bH}}\newcommand\XX{{\bX}}
\newcommand\ZZ{\mathbb{Z}}

\newcommand{\Ou}{\enm{\mathcal{O}}}

\nc{\J}{\enm{\mathcal{J} }}
\nc {\Z} {\enm{\mathbb{Z}}}
\nc {\form}[1] {\enm{\mbox{\underline{for}}}_{#1}}
\nc {\prol}[1] {\enm{\mbox{\underline{prol}}_{{#1}^*}}}

\nc {\stk} {\stackrel}

\newcommand{\map}{\rightarrow}

\newcommand{\inj}{\hookrightarrow}

\newcommand{\Pn}[2] {\ensuremath{ {\mathbb{P}}^{#1}_{#2}}}
\nc{\Quot}[3]{\enm{ {\mathfrak{Quot}_{ {#1}/{#2}/{#3}}}}}
\nc{\Hilb}[2]{\enm{ {\mathfrak{Hilb}_{ {#1}/{#2}}}}}
\newcommand{\mfrak}[1]{\mathfrak{#1}}
\newcommand{\mf}[1]{\mathfrak{#1}}
\newcommand{\bb}[1]{\mathbb{#1}}
\newcommand{\mcal}[1]{\mathcal{#1}}

\nc {\Coh}[4] {\ensuremath{H^{#1}(\Pn{#2}{},{#3}({#4}))}}
\nc {\Ch}[3] {\enm{H^{#1}(X_t,{#2}_t({#3}))}}
\nc {\Qphi}[4]{\enm{ {\mathfrak{Quot}^{~#4}_{ {#1}/{#2}/{#3}}}}}
\nc {\Gra}[4]{\enm{ {\mathfrak{Grass}_{#2}({#3},{#4})}}}
\nc {\HomA}[2]{\enm{\mathrm{Hom}_A{#1}{#2}}}
\nc {\tr}{\mathrm{tr}}

\nc {\C}[2]{\enm{\left(\begin{array}{l} {#1} \\ {#2} \end{array} \right)}}
\nc {\mat}[4]{\enm{\left(\begin{array}{ll}{#1} & {#2} \\ {#3} & {#4}
\end{array}\right)}}

\def \vp{\varphi}
\def \mb{\mbox}

 \def \Z{{\mathbb Z}}

   \def \h{\hat{\ }}

\def \d{\delta}   
   
\def \bA{{\mathbb A}}

\def \hG{\hat{\mathbb{G}}_{\mathrm{a}}} 

\def \hA{\hat{A}}

\def \R1{R((q))[q']\h}

\usepackage{xcolor}

\newcommand{\lam}{\lambda}

\DeclareMathOperator{\Spf}{\mathrm{Spf}}

\DeclareMathOperator{\Lie}{\mathrm{Lie}}

\newcommand{\Hom}{\mathrm{Hom}}
\newcommand{\End}{\mathrm{End}}
\newcommand{\Ext}{\mathrm{Ext}}

\newcommand{\teich}{\mathfrak{v}}

\newcommand{\longlabelmap}[1]{{\,\buildrel #1\over\longrightarrow\,}}
\newcommand{\longmap}{{\,\longrightarrow\,}}
\newcommand{\Fil}{\mathrm{Fil}}

\newcommand{\mff}{\mfrak{f}}

\nc{\tR}{\tilde{R}}
\nc{\bx}{\mathbf{x}}
\nc{\by}{\mathbf{y}}
\nc{\bz}{\mathbf{z}}
\nc{\ba}{\mathbf{a}}
\nc{\Fp}{\tilde{F}}
\nc{\Rp}{\tilde{R}}
\nc{\mlow}{m_{\mathrm{l}}}
\nc{\mup}{m_{\mathrm{u}}}
\nc{\ord}{{\mathrm{ord}}}
\nc{\bX}{\mathbf{X}}
\nc{\bH}{\mathbf{H}}
\nc{\bXp}{\mathbf{X}_{\mathrm{prim}}}
\nc{\bPsi}{\mathbf{\Psi}}
\nc{\mult}{\mathrm{mult}}
\nc{\mbB}{\mathbbm{B}}
\nc{\mfor}[1]{{#1}^{\mathrm{for}}}
\nc{\Hdr}{\bH^*_{\mathrm{dR}}}
\nc{\bt}{{\bf t}}
\nc{\beqar}{\begin{eqnarray*}}
\nc{\eeqar}{\end{eqnarray*}}
\nc{\fra}{\mfrak{f}}
\nc{\tW}{\tilde{W}}

\nc{\bo}{{\bf b}}
\nc{\hq}{{\hat{q}}}
\nc{\wHH}{\mathbf{H}_{\d}(E)}
\nc{\qH}{\mfrak{q}_{\wHH}}
\nc{\bfH}{\mathbf{H}(E)}
\nc{\admissible}{elliptic }
\nc {\rnk}{\mathrm{rank}}

\nc{\tF}{\tilde{F}}
\nc{\Nn}{N^{\mn}}
\nc{\del}{\Delta}
\nc{\tilW}{\tilde{W}}
\nc{\Di}[2]{\Delta^{#1}i^*\d^{#2}}
\nc{\di}[1]{i^*\d^{#1}}
\newcommand{\oE}{\overline{E}}
\nc{\pr}{\mathrm{pr}}
\nc{\Mat}{\mathrm{Mat}}

\nc{\tTheta}{\tilde{\Theta}}

\nc {\cblue}{\color{black}}

\DeclareRobustCommand{\loongleftarrow}{%
\leftarrow\joinrel\relbar\joinrel\relbar\DOTSB
}

\title
[Delta Theory of Anderson Modules II: Hodge--Pink structure]
{Delta Theory of Anderson Modules II: Hodge--Pink structure}

\author{Sudip Pandit}
\address{Department of Mathematics, 
King's College London, Strand, London WC2R 2LS, UK}
\email{sudip.pandit@kcl.ac.uk}

\author{Arnab Saha}
\address{Department of Mathematics, 
Indian Institute of Technology Gandhinagar, Gujarat 382355, India}
\email{arnab.saha@iitgn.ac.in}

\date{}



\begin{document}
\begin{abstract}
In this article, using the theory of $\d$-geometry, 
we construct a canonical  $z$-isocrystal 
$(\mathbf{H}_\delta(E), \mathfrak{f}^*)$ admitting a  Hodge-Pink structure for any abelian Anderson module $E$. The Hodge-Pink structure on $\mathbf{H}_\delta(E)$ induces a natural 
filtration 
$(\mathbf{H}_\delta(E) 
\supset \mathbf{X}_{\mathrm{prim}}(E)
\supset \{0\})$. 
The elements of $\mathbf{X}_{\mathrm{prim}}(E)$  are represented by primitive delta characters associated to $E$.

We establish a natural morphism from $\mathbf{H}_\delta(E)$ to the associated de Rham cohomology module $\mathbf{H}^*_{\mathrm{dR}}(E)$, which is strictly compatible with the aforementioned filtration and the classical Hodge filtration $(\mathbf{H}^{*}_{\mathrm{dR}}(E)\supset {\Lie(E)^{*}}\supset \{0\})$ on $\mathbf{H}^*_{\mathrm{dR}}(E)$.
Moreover, we show that the map induces an isomorphism between 
$\mathbf{X}_{\mathrm{prim}}(E)$ and $\mathrm{Lie}(E)^*$.  
Hence our isomorphism provides an interesting interpretation of the
 invariant differentials of $E$ as primitive delta characters of $E$.

Furthermore, when $E$ is a Drinfeld module, we show that the constructed $z$-isocrystal $\mathbf{H}_\delta(E)$ is weakly admissible. Consequently, the positive equal characteristic analogue of the Fontaine functor associates a crystalline $z$-adic Galois representation to the $\delta$-geometric object $\mathbf{H}_\delta(E)$. 
%It is also well known that the Tate module attached to $E$ gives rise to a natural Galois representation. 
In the case, when $E$ is the Carlitz module, we show that the Galois representation 
associated to $\mathbf{H}_\delta(E)$ is indeed the 
usual one coming from the Tate module.
\end{abstract}

\maketitle
\section{Introduction}

%\subsection{Background and Motivation}
Delta geometry developed by Buium draws inspiration from the theory of 
differential algebra over a function field 
and its applications in diophantine geometry
\cite{bui92, Buium94, gil02, Manin63}. 
%The foundations of delta geometry were established by Buium approximately three decades ago inspired by differential algebraic geometry and their applications in Diophantine geometry over function fields of characteristic zero 
The delta theory was subsequently developed in the category of $p$-adic formal schemes in a series of foundational articles \cite{Buium 1995, Buium 2000, Barc, Buium-Miller, BuSa1, BuSa2, Hurl}, yielding remarkable applications in diophantine geometry over number fields \cite{BP, Buium 1996, DP25, DP26}.

%Independently, Borger advanced the geometry of Witt vectors and developed algebraic jet spaces \cite{bor11a, bor11b} by functor of points. 

Borger and Saha in \cite{BS b} introduced delta geometry to the positive equal characteristic setting. Their work focused on Drinfeld modules, which are the function field analogs of elliptic curves introduced by Drinfeld to prove the Langlands conjecture for $\mathrm{GL}_2$ over global function fields \cite{Drinfeld_74, Drinfeld_80}. Given a Drinfeld module $E$, Borger and Saha constructed a $z$-isocrystal $(\bH(E),\mff^*)$ equipped with a natural map to the de Rham cohomology $\Hdr(E)$. By an explicit computation, in \cite{PS-3} we have shown that for Drinfeld modules of rank $2$, the $z$-isocrystal is non-degenerate (i.e, the semilinear operator $\mff^*$ is bijective) and proved a comparison theorem with the de Rham cohomology.  In \cite{PS-1}, we extended the delta theory to abelian Anderson modules (higher-dimensional generalizations of Drinfeld modules, cf.\ \cite{And, Hartl 2017}), which serve as the function field analogs of abelian varieties (we have recalled the main results of loc. cit. in\ Theorem \ref{finite_dim_thm}\cblue). The analogous theory in mixed characteristic, developed in a series of papers \cite{BS a, GPS, Pandit26, PS-2}, establishes deep connections with $p$-adic Hodge theory, culminating in comparison theorems with the crystalline cohomology of abelian schemes.

In parallel with Fontaine's foundational work in $p$-adic Hodge theory, Genestier--Lafforgue and Hartl developed equi-characteristic $p$-adic Hodge theory, establishing analogous structural results in \cite{GL2011, Hartl 2011}.

In the present article, building on the groundwork laid in \cite{BS b} and \cite{PS-1}, we construct a canonical $z$-isocrystal endowed with a Hodge--Pink structure associated to an abelian Anderson module, in the sense of \cite{GL2011, Hartl 2011, HJ 2020} (cf.\ Section \ref{LocalShtuka}). For Drinfeld modules, we prove that the associated $z$-isocrystal is weakly admissible. Consequently, the positive equal characteristic analog of the Fontaine functor yields a crystalline $z$-adic Galois representation (cf.\ Section \ref{weak-admissible}), which we compare with the classical Tate module Galois representation in favorable cases (cf.\ Section \ref{comparison}).

This paper serves as a sequel to \cite{PS-1} and maintains strict notational consistency with it. While we refer interested readers to the introductions of \cite{BS b} and \cite{PS-1} for a broader exposition on delta geometry in positive characteristic function-fields, we recall the essential terminology and notation below to ensure this article remains self-contained. 

\subsection{Notations and geometric setup}
Let $q$ be a fixed prime power and $\mathcal{C}$ be a projective, geometrically connected, smooth curve over $\mathbb{F}_{q}$. Fix an $\mathbb{F}_{q}$-rational point $\infty$ on $\mathcal{C}$. Let $A$ be the ring of functions regular outside $\infty$. Fix a maximal ideal $\mathfrak{p}$ of the Dedekind domain $A$ and an element $z\in \mathfrak{p}\setminus \mathfrak{p}^{2}$ of degree $f$ such that $A$ is a separable extension of $\FF_{q}[z]$. Note that it is always possible to find such a uniformizer for any $A$ (cf.\ page 18, Proposition 1.4 in \cite{Silverman}). 

Let $\hat{A}$ be the $\mathfrak{p}$-adic completion of $A$, and let $\pi$ be the image of $z$ in $\hat{A}$, which generates the maximal ideal $\hat{\mathfrak{p}}$. Let $k = A/\mathfrak{p}$ denote the finite residue field, and let $\hq=q^{f}$ be its cardinality. The quotient map $\hat{A}\longrightarrow \FF_{\hq}$ has a unique section, naturally endowing $\hat{A}$ with the structure of an $\FF_{\hq}$-algebra in addition to its $\mathbb{F}_{q}$-algebra structure.

Fix a flat $\hat{A}$-algebra $R$ (i.e., $\pi$-torsion free) that is a $\mathfrak{p}$-adically complete discrete valuation ring. We have an injection of rings:
$$\theta :A\longrightarrow \hat{A} \longrightarrow R$$
where $\theta$ is called the characteristic map. Let us also fix an $\hat{A}$-algebra endomorphism $\phi:R\longrightarrow R$ such that $\phi(x)\equiv x^{\hq}\bmod(\mathfrak{p}R)$ for $x\in R$. The affine formal scheme $\mathrm{Spf}\,R$ is denoted by $S$. Let $K$ be the fraction field of $R$ and $l$ be the residue field. By the Cohen structure theorem, we have $R\simeq l[[z]]$ and $K\simeq l((z))$. 

Given any $R$-module $M$, let $M_\phi$ denote the additive group $M$ 
with the $R$-module structure given by precomposition $R \stk{\phi}{\map}
R$. Also let $M_K$ denote the $K$-vector space $M \otimes_R K$.
Let $v$ denote the $\pi$-adic valuation on $R$, and let $\mathrm{ord}_z$ denote the $z$-adic valuation on $l((z))$. 

%\subsection{Anderson Modules and Prolongation Sequences}
Let $E$ be an abelian Anderson $A$-module of dimension $d$ and rank $r$ over $\mathrm{Spf}\,R$. The $A$-linear action of $z$ is given by 
$$\varphi_{E}(z)=\sum_{i=0}^{s}A_{i}{\tau}^{i}$$
where $A_{0}=\pi I+V(z)$ and $V$ is a nilpotent matrix. We define 
\begin{align*}
	V^{\perp}:=\{u\in \mathrm{Mat}_{1\times d}(R) \mid uV(z)=0\}
\end{align*}
Note that $V^{\perp}$ is a free $R$-submodule of $R^{d}$ of rank $h=d-\mathrm{rank}(V(z))$. 

For every positive integer $n$, following \cite{BS b}, we define the arithmetic jet space $J^{n}E$ of $E$ via its functor of points as 
$$J^{n}E(B):=E(W_{n}(B))$$
for any $\pi$-adically complete $R$-algebra $B$, where $W_n(B)$ is the $\pi$-typical Witt vectors of length $n+1$ \cite{drin76, BS b}. By \cite{bps, bor11b}
the functor $J^nE$ is representable by a $\pi$-formal scheme over $\Spf R$, 
which we continue to denote by $J^nE$.

This yields the short exact sequence of $A$-module schemes:
\begin{equation}\label{fshort1}
0 \map N^n \stk{i}{\map} J^nE \stk{u}{\map} E \map 0 
\end{equation}
where $u:J^nE \map E$ is the natural projection and $N^n = \ker(u)$. 

Let $\hG$ denote the $\pi$-formal additive group scheme over $\mathrm{Spf}\,R$. 
Note that $\hG$ naturally admits an $R$-module structure via scalar multiplication. Let $(\hG,\vp_{\hG})$ denote the $\pi$-formal additive group scheme $\hG$ equipped with the tautological $A$-action induced by the characteristic map $\theta:A \map R$. 

An $A$-linear morphism of group schemes from $J^{n}E$ to $\hat{\mathbb{G}}_{a}$ is called a \textit{differential} (or \textit{$\delta$-character}) of \textit{order $\leq n$} of $E$. We denote the group of all such $\delta$-characters by ${\bX}_{n}(E)$. Since $\hG$ is an $R$-module $\pi$-formal scheme over $S$, $\bX_n(E)$ is naturally an $R$-module. The inverse system $J^{n+1}E\xrightarrow{u} J^{n}E$ induces a directed system via pullback:
$$\ldots \xrightarrow{u^{*}} {\bX}_{n}(E)\xrightarrow{u^{*}} {\bX}_{n+1}(E) \xrightarrow{u^{*}}\ldots$$
We define the colimit as the $R$-module of $\delta$-characters
$${\bX}_{\infty}(E):=\varinjlim {\bX}_{n}(E).$$

The Frobenius morphism $\phi:J^{n+1}E \map J^nE$ on the jet spaces endows $\bX_\infty(E)$ with an $R\{\phi^*\}$-module structure, where $\phi^*$ is additive and $R$-semilinear, satisfying
$$\phi^*(r \Theta ) = \phi(r) \phi^*\Theta$$
and $\phi^*\Theta$ is the relative pullback of $\Theta$ induced by $\phi$.

In \cite{BS a, BS b}, the lateral Frobenius $\mfrak{f}:N^n \map N^{n-1}$ was constructed for all $n$, elevating $\{N^n\}_{n=1}^\infty$ into a prolongation sequence. This operator $\fra$ satisfies
\begin{equation}\label{comp}
\phi \circ i \circ \fra = \phi^{\circ 2} \circ i.
\end{equation}
Consequently, pulling back via $\mfrak{f}$ equips $\varinjlim \mathrm{Hom}_A(N^n,\hG)$ with an $R\{\mfrak{f}^*\}$-module structure. Following the principles of \cite{BS b}, we define the $R$-module:
$$\bH(E) := \varinjlim\frac{\mathrm{Hom}(N^n,\hG)}{i^*\phi^*(\bX_{n-1}(E)_{\phi})}.$$
By (\ref{comp}), the action of $\mfrak{f}^*$ on $\varinjlim \mathrm{Hom}(N^n,\hG)$ descends to $\bH(E)$. Consider the $R$-submodule 
$$\bXp(E):= \varinjlim \bX_n(E)/\phi^*(\bX_{n-1}(E)_{\phi})$$	
which fits into the short exact sequence:
\begin{equation}\label{Xpshort}
0 \map \bXp(E) \map \bH(E) \map \mathbf{I}(E) \map 0
\end{equation}
where ${\mathbf{I}}(E)$ is an $R$-submodule of $\mathrm{Ext}_A(E,\hG)$. Here, $\mathrm{Ext}_A(E,\hG)$ classifies the isomorphism classes of $A$-module $\pi$-formal schemes that are extensions of $E$ by $\hG$ (cf.\ Section 6 of \cite{PS-1}
and Section 5 of \cite{Gekeler_a}). 

Let $\bH^{*}_{\mathrm{dR}}(E)$ denote the de Rham cohomology module of $E$. In our previous article, we proved the following structural result:

\begin{theorem}[cf.~\cite{PS-1}]\label{finite_dim_thm}
For any abelian Anderson $A$-module $E$ of rank $r$ and dimension $d$, we have:
\begin{enumerate}
\item The $R$-module $\bH^{*}_{\mathrm{dR}}(E)$ is free of rank $r$.
\item The $R$-module $\mathrm{Ext}_A(E,\hG)$ is free of rank $r-h$. 
\item The $R$-module $\bXp(E)$ is free of rank $h$.
\item The $R$-module ${\bH}(E)$ is free of rank $\leq r$. 
\end{enumerate}
\end{theorem}

Furthermore, the theory of vectorial extensions (cf.\ Sections 6 and 8 in \cite{PS-1}) provides a natural map between the following short exact sequences:
\begin{equation}\label{diag-crys-limit-intro}
\xymatrix{
0  \ar[r] &{\bX}_{\mathrm{prim}}(E)\ar[d]^\Upsilon \ar[r] &{\bH}(E)\ar[d]^\Phi \ar[r] &\mathbf{I}(E)\ar@{^{(}->}[d] \ar[r]&0 \\
0 \ar[r] &\mathrm{Lie} (E)^{*} \ar[r] &\bH^{*}_{\mathrm{dR}}({E}) \ar[r] &\mathrm{Ext}_{A}(E,\hat{\mathbb{G}}_{a}) \ar[r] & 0.}
\end{equation}
The Hodge sequence in the above was also implicitly discussed in Proposition $3.1.3$ in \cite{BrPa} for $A=\FF_q[z].$ Note that even though the dimension of $\Lie E=d$ for an abelian Anderson module, its $A$-linear dual $\Lie (E)^* := 
\Hom_A(\Lie(E),\hG)$ might 
have rank less than $d.$ Precisely, we have proved the following theorem:
\begin{proposition} [\cite{PS-1}, Proposition $3.10$ ]Let $E$ be an abelian Anderson $A$-module of dimension $d$ and rank $r$ over 
$R$. The $R$-module $\Lie (E)^*$ is free of rank $h.$
\end{proposition} 

\subsection{Statement of main results}
An intriguing problem is to understand the precise relationship between the two short exact sequences in (\ref{diag-crys-limit-intro}). Towards this end, our first main result is the following:

\begin{theorem}\label{thm-1-intro}
Let $E$ be an abelian Anderson $A$-module of dimension $d$ and rank $r$ over 
$R$. 
The map $\Upsilon: \bXp(E) \map \mathrm{Lie} (E)^{*} $ is an injective morphism of $R$-modules with $\pi$-torsion cokernel. 

In particular, $\Upsilon_K: \bXp(E)_K \map \mathrm{Lie} (E)^{*}_K$ is an isomorphism of $K$-vector spaces.
\end{theorem}

Theorem 1.2 in \cite{PS-1} (recalled as Theorem \ref{finite_dim_thm} (4) above) 
establishes that $\bH(E)$ is a finite free module over $R\simeq l[[z]]$. Under this identification, the lift of Frobenius $\phi:l[[z]]\longrightarrow l[[z]]$ sends $z\mapsto z$ and $a\mapsto a^{q}$ for $a\in l$. We define the $l((z))$-vector space $\wHH:=\bH(E)\otimes_R l((z))$. The semilinear operator $\fra^*$ naturally extends to $\wHH$. 
Note that $(\wHH,\fra^*)$ becomes a $z$-isocrystal (definition recalled in 
Section \ref{LocalShtuka}) whenever $\det(\fra^*)$ is non-zero.

Our second main result ensures the non-degeneracy of the semilinear operator:

\begin{theorem}\label{thm-2-intro}
Let $E$ be an abelian Anderson $A$-module of dimension $d$ and rank $r$ over 
$R$. 
The semilinear operator $\fra^*: \bH(E)_K \map \bH(E)_K$ is a bijection. 

In particular, $(\wHH,\mathfrak{f}^{*})$ forms a $z$-isocrystal.
\end{theorem}

Let $e=\mathrm{ord}_{z}\left(\mathrm{det}(\mathfrak{f}^{*})\right).$ Recall that $\XX_{\mathrm{prim}}(E)=\langle \Theta_{1},\ldots,\Theta_{h} \rangle$, where the $\Theta_i$ are primitive characters of $E$. Extending this basis to $\wHH$, we obtain $\wHH=\langle i^{*}\Theta_{1},\ldots,i^{*}\Theta_{h},\Psi_{i_1},\dots , \Psi_{i_m}\rangle$ for certain $\Psi_{i_j}\in \bH(E)$. We define the Hodge--Pink structure on $\wHH$ as the following submodule of $\phi^*\bH(E) \otimes K((z-\pi))$:
\begin{align*}
	\qH:=\phi^{*}\langle(z-\pi)^{-e}i^{*}\Theta_{1},\ldots,(z-\pi)^{-e}i^{*}\Theta_{h},\Psi_{i_1},\dots \Psi_{i_m}\rangle\otimes_{l((z))} K[[z-\pi]].
\end{align*}

The details of the above construction of the Hodge-Pink structure are
elaborated in Section \ref{HP_delta}.

\begin{corollary} The object $(\wHH, \qH,\fra^*)$ is a $z$-isocrystal with Hodge--Pink structure.
\end{corollary}

By the positive equal characteristic analog of the Fontaine functor developed
by Genestier and Lafforgue \cite{GL2011} and Hartl \cite{Hartl 2011}, 
the weakly admissible $z$-isocrystals correspond to crystalline $z$-adic 
Galois representations. 
In this context, our next main result is the following: 

\begin{theorem}\label{thm-3-intro} 
Let $E$ be a Drinfeld module of rank $r$ over $R$.
\begin{itemize}
\item[(1)] Then the $z$-isocrystal with Hodge--Pink structure $(\wHH,\qH,\mathfrak{f}^{*})$ is weakly admissible. 
\item[(2)] In particular, $(\wHH, \qH, \fra^*)$ determines a crystalline Galois representation
$$\rho_{\wHH}: \mathrm{Gal}(K^{\mathrm{sep}}/K) \longrightarrow \mathrm{GL}_m(\FF_{\hq}((z)))$$
where $m = \dim_{l((z))} \wHH$ is the splitting number of $E$.
\end{itemize}
\end{theorem}
Recall from Section 9B of \cite{BS b} that for a Drinfeld module $E$, $\bH(E)$ is a free $R$-module of rank $m$, which is called the splitting number of $E$. To investigate how these Galois representations compare with those arising from the classical Tate modules associated with Drinfeld modules, we show that the representations coincide when $E$ has rank 1 (cf.\ subsection \ref{rk-1-com}  ). For higher rank cases, we prove the following relation between the characteristic
polynomials:

\begin{theorem}\label{thm-4-intro} 
Let $E/R$ be a Drinfeld module of rank $r$ with splitting number $r$. Then the characteristic polynomial of the linear operator $\mfrak{f}^{*}$ on $\bH(E)_K$ equals the Weil polynomial of the special fiber $\overline{E}.$
\end{theorem}

Beyond its intrinsic interest, Theorem \ref{thm-4-intro} hints at a deeper geometric connection when $E$ is a non-CL Drinfeld module of rank $2$, mirroring the theory of elliptic curves over mixed characteristics. In this scenario, the splitting number of $E$ is $2$, and by Theorem 1.1 in \cite{PS-3}, $\bH(E)_K$ is isomorphic to $\Hdr(E)_K$. Thus we have a natural Frobenius on $\Hdr(E)_K$ transported from $\bH(E)_K.$ In this context, Theorem \ref{thm-4-intro} implies that the characteristic polynomial of the induced Frobenius operator on $\Hdr(E)_K$ exactly matches the Weil polynomial of the special fiber.

Looking forward, the development of a suitable crystalline theory of universal vectorial extensions for Anderson modules—analogous to the Mazur--Messing framework \cite{MazMess}—will enable a general comparison theorem between the delta isocrystal and the  crystalline cohomology that would be compatible with their respective Hodge--Pink structures as done in \cite{Pandit26}.

{\bf Acknowledgements.} The first named author was partially supported by Vikram Sarabhai Research Fellowship at IIT Gandhinagar and Postdoctoral fellowship at IISER Mohali when most of the work was done, and Royal Society research fellowship 
 grant RF$\backslash $ERE$\backslash $231161.
The second author was partially supported by ANRF/ARGM/2025/001294/MTR.

\section{$\pi$-typical Witt vectors, $A$-module $\pi$-formal schemes and 
 arithmetic jet spaces}
In this section we recall some of the basic properties of $\pi$-typical Witt
vectors and  $A$-module 
$\pi$-formal schemes and the theory of delta characters associated to them.

\cblue
Let $q=p^{h}$ be a fixed prime power and $\mathcal{C}$ be a projective, geometrically connected, smooth curve over $\mathbb{F}_{q}$. Fix an $\mathbb{F}_{q}$-rational point $\infty$ on $\mathcal{C}$. Let $A$ be the ring of functions regular outside $\infty$. Fix a maximal ideal $\mathfrak{p}$ of the Dedekind domain
$A$ and $z\in \mathfrak{p}\setminus \mathfrak{p}^{2}$ be an element of degree $f$ such that $A$ is a separable extension of $\FF_{q}[z]$. Note that it is always possible to find such uniformizer for any $A$ (cf. page 18, Proposition $1.4$ in \cite{Silverman}). Let $\hat{A}$  be the $\mathfrak{p}$-adic completion  
of $A$ and $\pi$ be the image of $z$ in $\hat{A}$. Then $\pi$ generates the maximal ideal $\hat{\mathfrak{p}}$. Let $k$ denote the residue field $A/\mathfrak{p}$ which is finite and let $\hq=q^{f}$ be its cardinality. Note that the quotient map $\hat{A}\longrightarrow \FF_{\hq}$ has a unique section and thus $\hat{A}$ is not only an $\mathbb{F}_{q}$-algebra but also an $\FF_{\hq}$-algebra.
Let $\hat{A}$  be the $\mathfrak{p}$-adic completion  
of $A$ and $\pi$ be the image of $z$ in $\hat{A}$. 

For any $A$-algebra B and a $B$-algebra $C$, we define a {\it $\pi$-derivation}
$\delta$ as a set-theoretic map $\delta: B \map C$ that satisfies for all
$x,y \in B$, 
\begin{enumerate}
\item $\delta (1) = 0$ 

\item $\delta (x + y) = \delta x + \delta y$ 

\item $\delta (xy) = u(x)^{\hat{q}}\delta y + u(y)^{\hat{q}} \delta x +
\pi \delta x \delta y$,
\end{enumerate}
where $u:B \map C$ is the structure map.
Note that the pair $(u,\d)$ as above is uniquely characterised by a ring map 
from $B$ to $W_1(C)$, where $W_1(C)$ is the $\pi$-typical Witt vectors
of length $2$. Also note that since $R$ is of prime characteristic, 
the operator $\d$ is additive. However $\delta$ is not a ring map in general.

The associated {\it $\hat{q}$-power lift of Frobenius} map $\phi: B \map C$ is 
defined as $\phi(x):= u(x)^{\hat{q}} + \pi \delta x$ for all $x \in B$ 
and is easy to see that $\phi$ is a ring homomorphism.

Fix a flat $\hat{A}$-algebra $R$ which is a
$\pi$-adically complete discrete valuation ring. 
Let $R$ have a lift of Frobenius which is an $\hat{A}$-algebra
endomorphism $\phi:R\map R$ that satisfies $\phi(x) \equiv x^{\hat{q}}
\bmod (\mfrak{p}R)$. 
Associated to such a $\phi$, consider the unique $\pi$-derivation $\d:R \map
R$ given by $\d x = \frac{\phi(x)- x^{\hq}}{\pi}$. 

\subsection{$\pi$-typical Witt vectors}
We recall some of the basic theory of $\pi$-typical Witt vectors.
The general theory of Witt vectors over Dedekind domains with finite residue
fields were introduced by Borger \cite{bor11a}. The theory over local fields
of characteristic $p$ was introduced earlier by Drinfeld \cite{drin76}.

For an $R$-algebra $B$ with structure map $R \stk{f}{\map} B$, for all 
$n \geq 0$ let $B^{\phi^n}$ be the 
$R$-algebra with the structure map $R \stk{\phi^n}{\map} R \stk{f}{\map} B$.
We define the {\it ghost rings} $\prod_\phi^n B := B \times B^\phi \times 
\cdots \times B^{\phi^n}$ and $\prod_\phi^\infty B := B \times B^\phi \times 
\cdots $. For all $n \geq 1$, consider the {\it restriction} or 
{\it truncation} map on the ghost rings as $T_w : \prod_\phi^n B \map 
\prod_\phi^{n-1} B$ given by $T_w(w_0,\dots , w_n) := (w_0,\dots, w_{n-1})$.

Also consider the left shift {\it Frobenius} operators $F_w: \prod_\phi^n B
\map \prod_\phi^{n-1}B$ given by $F_w(w_0,\dots, w_n) = (w_1,\dots, w_n)$. 
Note that $T_w$ is an $R$-algebra map whereas $F_w$ lies over the fixed
Frobenius endomorphism $\phi$ of $R$.

Define as a set $W_n(B) := B^{n+1}$, and the set theoretic map
$w: W_n(B) \map \prod_\phi^n B$ by $w(x_0,\dots, x_n) := (w_0,\dots ,w_n)$
where for all $i \geq 0$,
$$
w_i = x_0^{{\hat{q}}^i} + \pi x_1^{{\hq}^{i-1}}+ \cdots + \pi^i x_i
$$
are the {\it Witt polynomials} and $w$ is known as the {\it ghost} map.
We define the ring of $\pi$-typical Witt vectors of length $n+1$ as the 
following analogous theorem to the $p$-typical case (also cf. pp. 171-192 in \cite{Mum-66}).

\begin{theorem}
For all $n \geq 0$, there exists a unique functorial $R$-algebra 
structure on $W_n(B)$ such that $w$ is a natural transformation of 
functors of $R$-algebras.
\end{theorem}

We now recall some of the important operators on the $\pi$-typical 
Witt vectors: 

$(1)$ The {\it restriction} or {\it truncation} map given by
$T(x_0,\dots,x_n) = (x_0,\dots, x_{n-1})$ and satisfies $w \circ T
= T_w \circ w$. 

$(2)$ The {\it Frobenius} map $F: W_n(B) \map W_{n-1}(B)$ satisfying
$w \circ F = F_w \circ w$ given by 
$$
F(x_0,\dots ,x_n) = (x_0^{\hq} + \pi x_1, \dots ).
$$
As in the case of the ghost side map, $F$ lies over the Frobenius
endomorphism $\phi$ of $R$.

Consider the Witt vectors of infinite length 
$W(B):= \varinjlim_T W_n(B)$ be the inverse limit taken over the
truncation map $T$. 

The map $F$ induces the lift of Frobenius map on 
$W(B)$. In fact there exists a unique delta map $\Delta: W(B) \map W(B)$
such that $F$ is the associated Frobenius map, that is 
$F(x) = x^{\hq} + \pi \Delta (x)$.

The ring $W(B)$ satisfies the following universal property: For any $R$-algebra
$C$ with a $\pi$-derivation $\delta$ on it and an $R$-algebra map $f:C \map B$,
there exists a unique $R$-algebra homomorphism $g: C \map W(B)$ such that 
the diagram 
$$
\xymatrix{
W(B) \ar[d]_T & \\
B & C \ar[l]^f \ar[lu]_g
}
$$
commutes and $g \circ \Delta = g \circ \delta$. Hence $W$ is the right adjoint
of the forgetful functor from $R$-algebras with $\pi$-derivation to $R$-algebra.
Hence via the above universal property of Witt vectors, for any $R$-algebra
$B$, $W(B)$ becomes naturally an $R$-algebra.

For more details, we refer the reader to Section $1$ of \cite{bor11a}. 
This
approach is analogous to that of Joyal in the case of $p$-typical Witt vectors
\cite{joyal}.

Given 
any $\pi$-formal scheme $X$ over $\Spf R$ the $n$-th jet space $J^nX$ is 
functorially defined as 
\begin{align}
\label{jetfuncdef}
J^nX(B) := X(W_n(B))
\end{align}
for all $\pi$-adically complete $R$-algebra $B$. Then $J^nX$ is representable 
by a $\pi$-formal scheme over $\Spf R$ that is obtained by glueing over open
affine covers of $X$- this is analogous to Buium's construction in the 
mixed characteristic case as in \cite{Buium 1995}.
For more details on the representability of the jet functor, the reader may
see \cite{bps} \cite{bor11b}. 

Let $X$ and $Y$ be $\pi$-formal schemes over $S$. We say a pair $(u,\d)$ is a 
{\it prolongation of $\pi$-formal schemes} and we write $Y\stk{(u,\d)}{\map} X$, if $u:Y \map X$ is a 
map of $\pi$-formal $S$-schemes and $\d: \Ou_X \map u_*\Ou_Y$ is a 
$\pi$-derivation making the following diagram commute:
$$\xymatrix{
R \ar[r] & u_*\Ou_Y \\
R \ar[u]^\d \ar[r] &\Ou_X. \ar[u]_\d
}
$$

We now recall the notion of prolongation sequences over 
$\pi$-formal schemes which are the positive characteristic analog of the
ones defined by Buium.
%For a more detailed treatment  we refer to \cite{BS_b,bui95}.
A {\it prolongation sequence of $\pi$-formal schemes} is a sequence 
$$
S \stk{(u_0,\d_0)}{\loongleftarrow} T^0 \stk{(u_1,\d_1)}{\loongleftarrow} 
T^1 \stk{(u_2,\d_2)}{\loongleftarrow} \cdots,
$$
where $T^i \stk{(u_{i+1},\d_{i+1})}{\loongleftarrow} T^{i+1}$ are 
prolongations satisfying
$$
u^*_i \circ \d_{i+1} = \d_i \circ u^*_{i+1}
$$
where $u^*_i$ is the pull-back morphism of sheaves induced by $u_i$
for each $i$. To simplify notations, we shall omit the subscripts on $u_i$
 and $\delta_i$ whenever no confusion is likely to arise.

We will denote a prolongation sequence of $\pi$-formal schemes
as $T^*$ or $\{T^n\}_{n\geq 0}$.
 Prolongation sequences form a category $\mcal{C}_{S^*}$, where a morphism $f:T^*\to U^*$ is 
a family of morphisms $f^n:T^n\to U^n$ commuting with both the $u$ and $\d$, in the evident sense.

Define $S^*=\{S^i\}_{i=0}^\infty$ to be the prolongation sequence given by 
$S^i := \Spf R$ $u_i := \mathbbm{1}$ and $\d_i = \d$ is the fixed 
$\pi$-derivation on $R$ for all $i$. This is the final object in $\mcal{C}_{S^*}.$

The system of $\pi$-formal schemes $J^*X:=\{J^nX\}_{n\geq 0}$ forms a 
prolongation sequence 
and is called the {\it canonical prolongation sequence}  
where $\phi_X: J^{n+1}X \map J^nX$ denote the lift of Frobenius
morphism for each $n$. Analogous to Proposition 1.1 in \cite{Buium 2000}, 
$J^*X$ satisfies the universal property that for any $T^* \in
\mcal{C}_{S^*}$ and $X$ a $\pi$-formal scheme over $S$ we have 
\begin{align}
	\label{univ}
\Hom_S(T^0,X) =\Hom_{\mathcal{C}_{S^*}}(T^*,J^*X).
\end{align}

\subsection{$A$-module schemes}
We will review some of the basic definitions and properties of $A$-module
schemes and their arithmetic jet spaces.

\begin{definition} An \emph{$A$-module scheme} is a pair $(E,\varphi_{E})$ consisting of a smooth $\pi$-formal group scheme $E$ over $S$ and a ring homomorphism $$\varphi_{E}:A\longrightarrow \mathrm{End}_{R}(E).$$ 
\end{definition}

Throughout this article, every $A$-module scheme is assumed to be defined over $S$ unless specified otherwise. Given two $A$-module schemes $G$ and $G^{'}$, we denote the set of all $\mathbb{F}_{q}$-linear morphisms of $S$-group schemes from $G$ to $G^{'}$ by $\mathrm{Hom}_{R,\mathbb{F}_{q}}(G,G^{'})$. \medskip

We now formalize the natural extension of the $A$-module structure from $E$ to $J^nE$. Let $\vp_E:A \map \mathrm{End}_R(E)$ denote the $A$-module structure on $E$. For any $a \in A$ and any $\pi$-adically complete $R$-algebra $C$, the functoriality of $E$ yields a group homomorphism $\vp_E(a): E(C) \map E(C)$.

We define the $A$-module structure on $J^nE$, denoted by $\vp_{J^nE}$, via $\vp_{J^nE}(a)(x):= \vp_E(a)(x)$ for all $x \in E(W_n(B)) = J^nE(B)$. This establishes a natural $A$-module structure on $J^nE$ induced by that of $E$. \cblue 

\cblue
\subsection{Explicit descriptions of $J^nE$}
\label{ExplicitJnE}
We now discuss the specific implications of the above definition of jet 
spaces to our setting
where $E = \hG^d = \Spf R\langle \bx \rangle$ over $\Spf R$ endowed with an 
$A$-module structure, where $\bx = (x_1,\dots , x_d)$ is tuple of 
$d$-indeterminates. This case will be especially pertinent to our cause when
$E$ is an admissible $A$-module (defined below).

Since $E = \hG^d$, by the definition
of the jet space functor in (\ref{jetfuncdef}) we have an isomorphism 
$J^nE \simeq \bb{W}_n^d$ 
where $\bb{W}_n$ is the affine $(n+1)$-space $\hat{\bA}^{n+1} \simeq 
R\langle \bx ,\bx_1, \dots ,\bx_n\rangle$ with the 
group structure of the additive $\pi$-typical Witt vectors of length $n+1$
and $\bx_i = (x_{1i},\dots , x_{di})$ is a tuple of $d$-indeterminates 
which are the {\it Witt coordinates} on $\bb{W}_n^d$ induced from $\bx$.
Also note that there exists another set of coordinates on $\bb{W}^d_n$, 
called the {\it Buium-Joyal } coordinates $(\bx, \bx', \dots 
\bx^{(n)})$ and one has a linear transformation between the above two choices
of coordinate functions for $\bb{W}_n^d$ such that we have
$$
\bb{W}_n^d \simeq \Spf R\langle \bx , \bx_1,\dots ,\bx_n\rangle 
\simeq \Spf \langle \bx ,\bx',\dots ,\bx^{(n)}\rangle.
$$
Hence the canonical prolongation sequence of $\pi$-formal rings (this is 
the natural notion associated to our definitions which were introduced above
for the case of $\pi$-formal schemes) associated to $\Ou(E) = R\langle
\bx \rangle$ is as follows:
$$\xymatrix{
\Ou(E) 
\ar[d]^-{{\rotatebox{90}{$\sim$}}} \ar[r]^-u_-\delta & \Ou(J^1E) 
\ar[d]^-{{\rotatebox{90}{$\sim$}}} \ar[r]^-u_-\delta & \cdots  
\ar[r]^-u_-\delta & \Ou(J^nE) \ar[d]^-{{\rotatebox{90}{$\sim$}}}
 \ar[r]^-u_-\delta & \cdots  \\
R\langle \bx \rangle \ar[r]^-u_-\delta & R\langle \bx, \bx'\rangle 
\ar[r]^-u_-\delta & \cdots  \ar[r]^-u_-\delta & R\langle \bx,\bx',\dots,
\bx^{(n)}\rangle \ar[r]^-u_-\delta & \cdots ,
}$$
where for all $n \geq 0$, $(u,\delta): \Ou(J^nE) \map \Ou(J^{n+1}E)$ is a 
$\pi$-derivation such that $u(\bx^{i}) = \bx^{i}$ and $\delta (\bx^{i})
= \bx^{i+1}$ for all $i = 0 , \dots n$. 

The associated lift of Frobenius map $\phi: \Ou(J^nE) \map \Ou(J^{n+1}E)$ 
is given by $\phi(a) = u(a)^{\hat{q}} + \delta a$ for all $a \in \Ou(J^nE)$
which is a ring homomorphism.

Here we remark that the ring map $u$ induces (at the level of 
$\pi$-formal schemes) the {\it truncation} or
the {\it restriction} map of Witt vectors $T: \bb{W}_n^d \map \bb{W}_{n-1}^d$
given by 
$$
T(\bx,\bx',\dots \bx^{(n)}) = (\bx, \dots, \bx^{(n-1)}).
$$
On the other hand, the ring map $\phi$ induces the {\it Frobenius} map 
of Witt vectors $F: \bb{W}_n^d \map \bb{W}_{n-1}^d$ given by
$$
F(\bx, \bx',\dots , \bx^{(n)}) = (\bx^{\hat{q}} + \pi \bx' ,\dots ,
{(\bx^{(n-1)})}^{\hat{q}} + \pi \bx^{(n)})
$$
and by the above discussion, $T$ and $F$ also are $A$-linear morphisms of 
$\pi$-formal schemes. 
For a detailed discussion on positive characteristic jet spaces, we refer the reader to Section $4$A of \cite{BS b}, and to \cite{bps, bor11b} for the general theory of arithmetic jet spaces. \color{black} 

Note that we have the following short exact sequence of $A$-module schemes:
\begin{align}
\label{fshort1}
0 \map N^n \stk{i}{\map} J^nE \stk{u}{\map} E \map 0 
\end{align}
where $u:J^nE \map E$ is the natural projection map and $N^n$ is the kernel of $u$. 

Given any $A$-module scheme $E$, consider the group $M(E):=\mathrm{Hom}_{R,\mathbb{F}_{q}}(E, \hat{\mathbb{G}}_{a})$ of $\mathbb{F}_{q}$-linear maps from $E$ to $\hat{\mathbb{G}}_{a}$ over $S$. The group $M(E)$ admits a natural left action by $A$ and a right action by $R$ defined by  
$$(a\otimes b)m:=b\circ m\circ \varphi_{E}(a)\quad \mathrm{for \ any \ a\otimes b \in A\otimes_{\mathbb{F}_{q}}R},$$
which corresponds to the composition
$$ E \xrightarrow{\varphi_{E}(a)}E\xrightarrow{m} \hat{\mathbb{G}}_{a}\xrightarrow{b}\hat{\mathbb{G}}_{a}.$$ 
The $A\otimes_{\mathbb{F}_{q}}R$-module $M(E)$ is called the $A$-\emph{motive} associated to $E$.  
\cblue
Note that an $A$-module scheme $E$ naturally endows $\Lie (E)$ with an $A$-module structure. This structure, denoted by $\vp_{\Lie (E)}$, is given by the derivative map $D\vp_E(a)$ for all $a \in A$; that is, 
$$
\vp_{\Lie (E)}(a)(y) := D\vp_E(a)(y)
$$
for all $a \in A$ and $y \in \Lie (E)$.

\color{black}

\begin{definition}  
\label{abAnddef}
An abelian Anderson $A$-module of rank $r$ and dimension $d$ over $S$ is an affine $A$-module scheme $(E,\varphi_{E})$ over $S$ of relative dimension $d$ such that:
\begin{itemize}
\item[(i)] $E$ is isomorphic to $\hat{\mathbb{G}}_{a}^{d}$ as $\mathbb{F}_{q}$-modules.
\item[(ii)] $(\varphi_{\mathrm{Lie}E}(a)-\theta(a))^{d}=0$, for all $a \in A$.
\item[(iii)] The associated $A$-motive $M(E)$ is a locally free $A\otimes_{\mathbb{F}_{q}}R$-module of rank $r$.
\end{itemize}
\end{definition}
\noindent We remark that even if $E$ is an Anderson module, $J^nE$ is not necessarily one.

We say $E$ is an \emph{admissible} $A$-module of dimension $d$ if it satisfies conditions (i) and (ii) above. Note that if $E$ is an admissible $A$-module, then $\Lie(E)$ is also admissible, as the $A$-action is given by $\varphi_{\Lie (E)}(z)=A_{0}{\tau}^{0}.$ \cblue 
Let $E$ be an admissible $A$-module of dimension $d$ over $S$ with the $A$-action given by 
 $$\varphi_{E}(z)=\displaystyle{\sum_{i=0}^{s}}A_{i}{\tau}^{i},$$
where $A_{0}=\pi I+V(z)$. Observe that $\varphi_{E}(ab)=\varphi_{E}(a)\circ \varphi_{E}(b)$. Consequently, we obtain:
\begin{align}\label{V-id}
\nonumber \theta(ab)I+V(ab)&=(\theta(a)I+V_{a})(\theta(b)I+V(b))\\
V(ab)&=\theta(a)V(b)+\theta(b)V(a)+V(a)V(b).
\end{align}
This implies, in particular, that $U \circ V(z)=0$ yields $U \circ V(g)=0$ for all $g\in \FF_{q}[z]$, where $U$ is an $R$-linear map. Define 
\begin{align*}
	V^{\perp}:=\{u\in \mathrm{Mat}_{1\times d}(R)~|~uV(z)=0\}.
\end{align*}
Since $R$ is a discrete valuation ring, $V^{\perp}$ is a free submodule of $R^{d}$ of rank $h:=d-\mathrm{rank}(V(z))$.

\subsection{Delta characters associated to $A$-module schemes}
Let $\hG$ denote the $\pi$-formal additive group scheme over $\Spf R$. 
Note that $\hG$ naturally admits an $R$-module structure via scalar multiplication. Let $(\hG,\vp_{\hG})$ denote the $\pi$-formal additive group scheme $\hG$ equipped with the tautological $A$-action induced by the characteristic map $\theta:A \map R$. 

Let $E$ be an $A$-module scheme. We denote by ${\bX}(E):=\mathrm{Hom}_{A}(E,\hat{\mathbb{G}}_{a})$ the group of all $A$-linear morphisms of $\pi$-formal group schemes from $E$ to $\hat{\mathbb{G}}_{a}$, where $\hat{\mathbb{G}}_{a}$ is equipped with the $A$-module structure induced by the characteristic map. 
 %For convenience, we will denote ${\bX}(J^{n}E)$ by ${\bX}_{n}(E)$ throughout the remainder of this article. 

An $A$-linear morphism of group schemes from $J^{n}E$ to $\hat{\mathbb{G}}_{a}$ is called a \textit{differential} or \textit{delta character of order $\leq n$} of $E$. The group of all such $\d$-characters is denoted by ${\bX}_{n}(E)$.

Since $\hG$ is an $R$-module $\pi$-formal scheme over $S$, $\bX_n(E)$ naturally becomes an $R$-module. Furthermore, the inverse system $J^{n+1}E\xrightarrow{u} J^{n}E$ defines a directed system of $R$-modules via the pullback map:
$$\ldots \xrightarrow{u^{*}} {\bX}_{n}(E)\xrightarrow{u^{*}} {\bX}_{n+1}(E) \xrightarrow{u^{*}}\ldots$$ 
We define 
$${\bX}_{\infty}(E)=\varinjlim {\bX}_{n}(E).$$

Similarly, pulling back by the Frobenius map $\phi:J^{n+1}E\to J^nE$ induces a Frobenius operator
$\bX^n(E)\to \bX^{n+1}(E)$. However since $\phi:J^{n+1}E\to J^nE$ is not a morphism over $\Spf R$ but
instead lies over the Frobenius endomorphism $\phi$ of $\Spf R$, some care is required.
Consider the relative Frobenius morphism $\phi_{E/R}$, defined to be the unique
morphism making the following diagram commute:
$$\xymatrix{
J^{n+1}E \ar@{.>}^{\phi_{E/R}}[rd] \ar@/^/[rrd]^\phi \ar@/_/[ddr]& & \\
& J^nE \times_{(\Spf R),\phi} \Spf R \ar[d] \ar[r] & J^nE \ar[d] \\
& \Spf R \ar[r]_\phi & \Spf R 
}$$
Then $\phi_{E/R}$ is a morphism of $A$-module formal schemes over $\Spf R$.

Then given a $\d$-character $\Theta:J^nE\to \hG$, define $\phi^*\Theta$ to be 
the composition
\begin{equation}
	J^{n+1}E \longlabelmap{\phi_{E/R}} J^nE \times_{(\Spf R),\phi} \Spf R 
	\longlabelmap{\Theta\times\mathbbm{1}} \hG\times_{(\Spf R),\phi} \Spf R\longlabelmap{\iota} \hG,
\end{equation}
where $\iota$ is the isomorphism of $A$-module formal schemes over $R$
coming from the fact that $\hG$ descends to $\hat{A}$ as an $A$-module scheme. For any $R$-algebra $B$,
the induced morphism on $B$-points is
$$
E(W_{n+1}(B)) \longlabelmap{E(F)} E(W_n(B)^\phi) \longlabelmap{\Theta_B^\phi} 
B^\phi \longlabelmap{b\mapsto b} B.
$$
Note that the composition $E(W_{n+1}(B)) \map B$ is indeed a morphism of 
$A$-modules because the identity map $B^\phi \map B$ is $A$-linear since 
$\phi$ restricted to $\hA$ is the identity.

Hence we have an additive map $\phi^*: \bX_n(E) \to \bX_{n+1}(E)$ given by 
$\Theta\mapsto \phi^*\Theta$. Note that this map is not $R$-linear. 
 However, the map
$$\bX_n(E) \longmap \bX_{n+1}(E)^\phi, \quad \Theta\mapsto 
\phi^*\Theta $$ is $R$-linear, where for any $R$-module $M$, $M^\phi$ 
is the group $M$ endowed with the $R$-action obtained by precomposing
with the Frobenius $\phi: R \map R$. 
Taking direct limits in $n$, we obtain an $R$-linear map
$$ \bX_\infty(E) \longmap \bX_\infty(E)^\phi, \quad \Theta\mapsto
\phi^*\Theta.  $$
Hence $\bX_\infty(E)$ is a left module over the twisted polynomial
ring $R\{\phi^*\}$ with commutation law $\phi^* r = \phi(r)\phi^*$.

\subsection{Splitting of $J^{n}E$}\label{splitting_J^n}
The short exact sequence \eqref{fshort1} always splits as $\mathbb{F}_{q}$-module schemes via the \cblue \textup{Teichm\"{u}ller section}$$\teich:E\longrightarrow J^{n}E .$$ Thus, as an $\mathbb{F}_{q}$-module scheme, $$J^{n}E \simeq E\times N^{n}.$$ For any element $\Theta \in {\bX}_{n}(E)$, we have the following commutative diagram: 

$${
\xymatrix{
E \ar[r]^\teich \ar[rd]_{\teich^{*}\Theta} &
J^{n}E \ar[d]^\Theta &
N^{n} \ar[l]_\iota \ar[ld]^{\iota^{*}\Theta}\\
&\hat{\mathbb{G}}_a}
}.$$
This allows us to write $\Theta(\mathbf{x}_{0},\ldots \mathbf{x}_{n})=\teich^{*}\Theta(\mathbf{x}_{0})+\iota^{*}\Theta(\mathbf{x}_{1},\ldots ,\mathbf{x}_{n})$. We emphasize that the splitting of $J^{n}E$ is not, in general, an isomorphism of $A$-module schemes, as the map $\teich^{*}\Theta$ is only $\mathbb{F}_q$-linear. The following theorems provide a straightforward generalization of Section $8$ in \cite{BS b}. This splitting induces a splitting in the Lie algebra of the jet sequence \ref{fshort1}, which would be crucial in Section \ref{ziso} relating delta characters to de Rham cohomology.

\subsection{Character groups of kernels and their explicit bases} In this subsection, we review the construction of the lateral Frobenius morphisms and explicitly describe the characters of the  kernels of jet spaces.\subsubsection{Lateral Frobenius}We begin by recalling the lateral Frobenius morphisms as constructed in \cite{BS b}. Let $\mathbb{W}_{n}$ denote the $\pi$-formal group scheme, which is isomorphic to $\hat{\mathbb{A}}^{n+1}$, equipped with the group structure of the additive Witt vectors of length $n+1$. Consider the following commutative diagram:
\begin{equation}\label{Lat_witt}
{
\xymatrix{
\mathbb{W}_{n}\ar[d]_{F} \ar[r]^{V} &\mathbb{W}_{n+1}\ar[d]^{F}  \\
\mathbb{W}_{n-1} \ar[r]^{V} &\mathbb{W}_{n}\ar[d]^{F}\\
&\mathbb{W}_{n-1} }}
\end{equation}
satisfying the following relation  
\begin{equation}\label{Fr_var}
FFV=FVF.
\end{equation}
which satisfies the fundamental relation:
\begin{equation}\label{Fr_var} FFV=FVF.
\end{equation}
This relation holds because the operator $FV$ acts as multiplication by $\pi$, and the Frobenius $F$  commutes with $\pi$. We can recast this diagram in the setting of jet spaces by  the natural group scheme identifications $J^{n}E\simeq \mathbb{W}_{n}^{d}$ and $N^{n}\simeq \mathbb{W}_{n-1}^{d}$. Within this context, we adopt the notation $i := V$ and $\phi := F$ for the morphisms in the right-hand column. We then define the \emph{lateral Frobenius}
\begin{equation*}\mathfrak{f}:N^{n+1}\rightarrow N^{n}\end{equation*}
to be the morphism $F:\mathbb{W}_{n}\rightarrow \mathbb{W}_{n-1}$ operating under these identifications. 
Consequently, Diagram \eqref{Lat_witt} translates to the following:
\begin{equation}\label{Lat_jet}
{
\xymatrix{
N^{n+1}\ar[d]_{\mathfrak{f}} \ar[r]^{i} &J^{n+1}E\ar[d]^{\phi}  \\
N^{n} \ar[r]^{i} &J^{n}E\ar[d]^{\phi}\\
&J^{n-1}E }}
\end{equation} It is crucial to emphasize that whenever we employ the notation $N^{n}$, the $A$-module structure is intrinsically understood to be the one that renders $i$ an $A$-linear morphism. This must be strictly distinguished from the $A$-module structure induced via transport of structure from the group scheme isomorphisms $N^{n}\simeq \mathbb{W}_{n-1}^{d} \simeq J^{n-1}E$. Furthermore, we remark that while $i$ operates as a morphism of $S$-schemes, the vertical arrows $\phi$ and $\mathfrak{f}$ in Diagram \eqref{Lat_jet} lie over the Frobenius endomorphism of $S$, rather than the identity morphism.

Although the lateral Frobenius $\mathfrak{f}:N^{n}\rightarrow N^{n-1}$ is defined a priori as a morphism of group schemes, an argument entirely analogous to Proposition 7.2 in \cite{BS b} demonstrates that it is, in fact, $A$-linear. The lateral Frobenius $\mathfrak{f}$, paired with the natural projection map $u$, endows the system $N^*E:=\{N^{n}\}_{n\geq 0}$ with the structure of a prolongation sequence. This leads to the following structural result:
\begin{theorem}\label{affjet1}If $E$ is an admissible $A$-module scheme over $S$, then we have$$N^{*}E \simeq J^*(N^1E)$$as a canonical isomorphism of prolongation sequences of $A$-module schemes over $S^*$.In particular, for all $n \geq 1$, the above induces a canonical isomorphism$$N^{n}E \simeq J^{n-1}(N^{1}E)$$of $A$-module schemes over $S$.\end{theorem}
\begin{proof} The proof follows directly from Theorem 1.2 in \cite{PS-2}.\end{proof} 
Recall the following fundamental result regarding the character group of $N^1$:
\begin{lemma}\label{rank-N^1} Let $E$ be an abelian Anderson $A$-module of dimension $d$ and rank $r$ over $R$. Then the rank of the $R$-module $\bX(N^{1})$ is $h$.\end{lemma}
\begin{proof} See Corollary 3.11 in \cite{PS-1}.\end{proof}

Let us fix an $R$-linear basis $\{B_{01}, B_{02}, \ldots, B_{0h}\}$ for $V^{\perp} \simeq \Lie (E)^*$ (we refer the reader to Proposition 3.10 in \cite{PS-1} for details on this $R$-module isomorphism). In the degenerate case where $V$ vanishes, $V^{\perp}=R^d$ identically, and we may naturally select the standard basis $\{e_1,\ldots, e_{d}\}$. Let $\{\Psi_{11},\Psi_{12},\ldots, \Psi_{1h}\}$ denote an $R$-linear basis of the character group $\bX(N^{1}):=\Hom_A(N^1,\hG)$ such that $D\Psi_{1j}=\pi^{\nu}B_{0j}$, where $\nu$ is the  least positive integer for which $\Psi_{1j}$ is defined over $R$ (see Theorem 3.9 in \cite{PS-1}) for all $1\leq j\leq h$. We encapsulate this basis into the tuple $\Psi_1:=(\Psi_{11},\Psi_{12},\ldots, \Psi_{1h}) \in \Hom_A(N^1,\hG^h)$. For every pair of positive  integers $n\geq i\geq 1$ and $1\leq j\leq h,$ we define the composition map:$$\Psi_{ij}:=\Psi_{1j}\circ\mff^{\circ i}: N^{n}\to \hG,$$where $\mff^{\circ i}$ denotes the sequential composition $\mff^{\circ i}: N^n\xrightarrow{\mff\circ \ldots \circ\mff} N^{n-i}\xrightarrow{u} N^1$. We gather these characters into the tuple $\Psi_i:=(\Psi_{i1},\Psi_{i2},\ldots, \Psi_{ih}) \in \Hom_A(N^n,\hG^h)$.

Following a similar argument to Proposition 7.4 in \cite{BS b}, and by Theorem $4.2$ in \cite{PS-1} we derive the following lemma.
\begin{lemma}\label{ind-psi} Let $E$ be an abelian Anderson $A$-module. Then the elements $\{\Psi_{11}, \ldots,\Psi_{1h},\ldots,  \Psi_{n1},\ldots \Psi_{nh}\}$ are $R$-linearly independent in $\Hom_A(N^n, \hG)$. Hence $\{\Psi_1, \ldots, \Psi_n\}$ forms a $K$-basis of  $\Hom_A(N^n, \hG^h)_K.$\end{lemma}
\begin{proof}We first show that any $R$-linear relation among the elements $\{\Psi_1, \ldots, \Psi_n\}$ is trivial. We establish linear independency by examining the lowest-degree terms of the coordinate expansions.

In Buium--Joyal coordinates as in section \ref{ExplicitJnE}, we may express $\Psi_{1}$ as$$\Psi_{1}(\bx') = \pi^{\nu}B_{0}\bx' + B_1(\bx')^q + \cdots,$$where $B_{0} = [B_{01}, \dots, B_{0h}]$ is the coefficient matrix corresponding to the basis of $\Lie (E)^*$. Recall that the lateral Frobenius $\mff$ acts as the $\hq=q^f$-th power Frobenius modulo $\pi$. Indeed, we have
\begin{align*}\mff: N^2&\to N^1\\
\mff(\bx',\bx'')&=\bx'^{(\hq)}+\pi \bx''
\end{align*}
where $M^{(q)} = \left(M^{q}_{ij}\right)$ denotes the matrix obtained by raising each entry of a matrix $M \in \mathrm{Mat}_{m\times n}(R)$ to the $q$-th power. 
Consequently, the composition $\Psi_{2} = \Psi_1 \circ \mff$ has the coordinate expansion
\begin{align*}\Psi_{2}(\bx',\bx'') &=( \mff^*\Psi_1)(\bx',\bx'')\\
 &=(\pi^\nu B_0)^{\phi}(\bx'^{(\hq)})+\pi \bx'')+\ldots\\
 &=\pi^{\nu} B_{0}^{(\phi)}(\bx')^{(\hq)} + \text{higher order terms},
\end{align*}
where $M^\phi:=\phi(m_{ij})$ for any matrix $M=(m_{ij}) \in \mathrm{Mat}_{m\times n}(R).$ 
By inductive argument we get,
$$\Psi_{i}(\bx',\bx'',\ldots) = =\pi^{\nu} B_{0}^{\phi^{(i-1)}}(\bx')^{(\hq^{i-1})} + \text{higher order terms}$$

Suppose there is an $R$-linear dependence relation $\sum_{i=1}^n c_i \Psi_i = 0$. Because the lowest-degree non-zero terms of each $\Psi_i$ appear at distinct degrees $\hq^{i-1}$, they cannot cancel one another. Thus, the relation implies that the coefficient of $(\bx')^{(\hq^{i-1})}$ must vanish identically for each $i$. To ensure $c_i = 0$, it suffices to show that the matrix $B_0^{\phi^{(i-1)}}$ is non-zero, which is true as $\phi$ is an automorphism of $R$ and $B_0$ is a non-zero matrix. 

To prove $\{\Psi_{11}, \ldots,\Psi_{1h},\ldots,  \Psi_{n1},\ldots \Psi_{nh}\}$ are $R$-linearly independent in $\Hom_A(N^n, \hG)$, it is enough to to prove that $\sum_{j=1}^h c_{ij} \Psi_{ij} = 0$ implies $c_{ij}=0,$ for $c_{ij}\in R.$
Moreover, considering the coefficients of $\bx'^{(\hq^{i-1})}$, it is enough to prove that the column vectors $\{B_{01}^{\phi^{(i-1)}}, \dots, B_{0h}^{\phi^{(i-1)}}\}$ are $R$-linearly independent.

Given that the elements $\{B_{01}, B_{02}, \ldots, B_{0h}\}$ are $R$-linearly independent, meaning there exists an $h \times h$ minor of $B_0$ with a non-zero determinant. Since $\phi$ is a ring homomorphism on $R$, for any square matrix $B$ with entries in $R$, the identity $\det(B^{\phi^{(i-1)}}) = \phi^{(i-1)}(\det B)$ naturally holds. Moreover, as $\phi$ is an automorphism, applying this to a non-singular $h \times h$ minor of $B_0$, we deduce that the corresponding $h \times h$ minor of $B_0^{\phi^{(i-1)}}$ is also non-singular. This implies the column vectors $\{B_{01}^{\phi^{(i-1)}}, \dots, B_{0h}^{\phi^{(i-1)}}\}$ remain $R$-linearly independent, completing the first assertion of the proof. 

By Theorem $4.2$ in \cite{PS-1} we can derive that $\dim_K \Hom(N^n,\hG^h)_K$ is $n.$ Thus the $R$-linearly independent set $\{\Psi_{1}, \dots, \Psi_{n}\}$ forms a $K$-basis of $\Hom_A(N^n, \hG^h)_K$, completing the proof.\end{proof}

\noindent \textbf{Characters of the Carlitz Module.} Let $A=\FF_q[z]$. The $A$-module structure defining the Carlitz module $E$ is given by:$$\varphi_{E}(z)=\pi z+z^q.$$Following the construction detailed in \cite{BS b} (p. 821), we obtain a canonical character, expressed in local coordinates as:$$\Psi_1(x')=1/\pi \log_C(\pi(x-x')),$$where $\log_C$ denotes the formal Carlitz logarithm.

\subsection{The $R$-module $\bXp(E)$ and ${\bH}(E)$}\label{ziso}
In this subsection, we recall specific constructions from \cite{BS b, PS-1}. 
The aim is to explain the map between short exact sequences as in
 (\ref{diag-crys-limit-intro}).
These will be utilized later in Section \ref{LocalShtuka} to define a $z$-isocrystal equipped with a Hodge-Pink structure canonically associated to an Anderson module $E$.

A $\d$-character $\Theta \in {\bX}_{n}(E)_{K} $ is called \emph{primitive} if $$\Theta \notin \bX(E)_{K}+ u^{*}{\bX}_{n-1}(E)_{K}+\phi^{*}({\bX}_{n-1}(E)_{K})_{\phi}.$$
For $i\geq 1$, let $\mathbb{B}_{i}\subset {\bX}_{i}(E)_{K}$ be a finite subset whose elements project to distinct, linearly independent vectors that form a $K$-basis for $$\dfrac{{\bX}_{i}(E)_{K}}{u^{*}{\bX}_{i-1}(E)_{K}+\phi^{*}({\bX}_{i-1}(E)_{K})_{\phi}}.$$ Such a set $\mathbb{B}_{i}$ is called a \emph{primitive basis} of ${\bX}_{i}(E)_{K}$. We define the module of primitive characters as $${\bX}_{\mathrm{prim}}(E):=\underrightarrow{\lim} \dfrac{{\bX}_{n}(E)}{\bX(E)+\phi^{*}({\bX}_{n-1}(E))_{\phi}}.$$ 

Furthermore, we define the $R$-module $${\bH}_{n}(E):=\dfrac{{\bX}(N^{n})}{i^{*}\phi^{*}({\bX}_{n-1}(E)_{\phi})}.$$
The morphism $u:N^{n+1}\longrightarrow N^{n}$ induces a map $u^{*}:{\bX}(N^{n})\longrightarrow {\bX}(N^{n+1})$, which commutes with both $i^{*}$ and $\phi^{*}$. Consequently, we obtain a map $$u^{*}:{\bH}_{n}(E)\longrightarrow {\bH}_{n+1}(E).$$
We then define ${\bH}(E):=\underrightarrow{\lim} \, {\bH}_{n}(E)$ and $\mathbf{I}(E):=\underrightarrow{\lim} \, \mathbf{I}_{n}(E)$.
Similarly, $\mathfrak{f}:N^{n+1}\longrightarrow N^{n}$ induces a linear map $\mathfrak{f}^{*}:{\bX}(N^{n})\longrightarrow {\bX}(N^{n+1})$, which descends to a semilinear map of $R$-modules $$\mathfrak{f}^{*}:{\bH}_{n}(E)\longrightarrow {\bH}_{n+1}(E)$$ by Proposition \ref{diff} (ii). This induces a semilinear endomorphism $\mathfrak{f}^{*}:{\bH}(E)\longrightarrow {\bH}(E)$. 

Recall from Section $5$ of \cite{Gekeler_a} that $\mathrm{Ext}_{A}(E,\hG)$ denotes the group of extension classes of $E$ by $\hG$ as $A$-module schemes over $R$. Moreover, $\mathrm{Ext}^{\sharp}_{A}(E,\hG)$ denotes the group of extension classes of $E$ by $\hG$ over $R$, equipped with a splitting of the induced extension of Lie algebras. By definition, for an abelian Anderson module $E$, we have $\Hdr(E)=\mathrm{Ext}^{\sharp}_{A}(E,\hG)$.

Consider the exact sequence $$0\longrightarrow N^{n}\longrightarrow J^{n}E\longrightarrow E\longrightarrow 0.$$ 
Applying $\HomA(-,\hG)$ to the above short exact sequence of $A$-module 
$\pi$-formal schemes we obtain the following long exact sequence of 
$R$-modules
\begin{align}
\label{connecting}
0\map \bX(E) \map \bX_n(E) \map \bX(N^n) \stk{\partial}{\longrightarrow} 
\Ext(E,\hG)
\end{align}
where $\partial$ is the connecting morphism of $R$-modules.

The Teichmüller section $v:E\longrightarrow J^{n}E$ induces a splitting of the corresponding Lie algebras:
\begin{align*}
 \xymatrix{
0  \ar[r] &\mathrm{Lie} (N^{n})\ar[r] &\mathrm{Lie} (J^{n}E) \ar[r] \ar@/^1pc/[l]^{s_{\mathrm{Witt}}} &\mathrm{Lie} (E)  \ar[r]&0}
\end{align*}
where $s_{\mathrm{Witt}}(\mathbf{x}_{0},\ldots,\mathbf{x}_{n})=(\mathbf{x}_{1},\ldots \mathbf{x}_{n})$ is the associated retraction. For any $\Psi \in {\bX}(N^{n})$, we obtain the following pushout diagram: 
\begin{align*}
\xymatrix{
0  \ar[r] &N^{n} \ar[d]^{\Psi} \ar[r] &J^{n}E \ar[d]^{e_{\Psi}} \ar[r] &E \ar@{=}[d] \ar[r]&0 \\
0 \ar[r] &\hat{\mathbb{G}}_{a} \ar[r] &E_{\Psi}^{*} \ar[r] & E \ar[r] &0}
\end{align*}
where $E_{\Psi}^{*}=(J^{n}E\times \hat{\mathbb{G}}_{a})/\Gamma(N^{n})$ and $\Gamma(N^{n})=\{i(\mathbf{z}),-\Psi(\mathbf{z})~|~\mathbf{z}\in N^{n}\}\subset J^{n}E\times \hat{\mathbb{G}}_{a}.$
Let $s_{\Psi}$ denote the induced splitting of Lie algebras for the pushout extension:
\begin{align*}
 \xymatrix{
0  \ar[r] &\mathrm{Lie} (\hat{\mathbb{G}}_{a})\ar[r] &\mathrm{Lie} (E_{\Psi}^{*}) \ar[r]\ar@/^1pc/[l]^{s_{\Psi}} &\mathrm{Lie} (E)  \ar[r]&0}.
\end{align*}
We define $\tilde{s}_{\Psi}:\mathrm{Lie}(J^{n}E)\times \mathrm{Lie}(\hat{\mathbb{G}}_{a})\longrightarrow \mathrm{Lie}(\hat{\mathbb{G}}_{a})$ as
\begin{align*}
\tilde{s}_{\Psi}(\mathbf{x},y)=D\Psi(s_{\mathrm{Witt}}(\mathbf{x}))+y, 
\end{align*}
from which it follows that $\tilde{s}_{\Psi}|_{\mathrm{Lie}(\Gamma(N^{n}))}=0$. Therefore, 
\begin{align}\label{s_Witt_def}
s_{\Psi}[\mathbf{x},y]=D\Psi(s_{\mathrm{Witt}}(\mathbf{x}))+y 
\end{align}
is a well-defined retraction.

Hence the above splitting and (\ref{connecting}) we obtain the following 
commutative diagram of $R$-modules:
\begin{equation}
\xymatrix{
		0 \ar[r] & {\bX}_{n}(E)/\bX(E) \ar[r] \ar[d]^{\tilde{\Upsilon}} & 
		{\bX}(N^{n}) \ar[r] \ar[d]^{\tilde{\Phi}}& 
		\mathrm{Ext}_{A}(E,\hat{\mathbb{G}}_{a})  \ar@{=}[d]& \\
		0 \ar[r] &\mathrm{Lie} (E)^{*} \ar[r] & \mathrm{Ext}^{\sharp}_{A}(E,\hat{\mathbb{G}}_{a}) \ar[r] & 
\mathrm{Ext}_{A}(E,\hat{\mathbb{G}}_{a}) \ar[r] & 0,
	}
\end{equation}\label{Char_to_Hdg_fil_1}
where $\tilde{\Phi}(\Psi) = (E^*_\Psi,s_\Psi) \subset \bX(N^n)$ 
for all $\Psi \in \bX(N^n)$
and $\tilde{\Upsilon}$ is the naturally induced morphism.

It follows that the submodule $\iota^* \phi^{*}({\bX}_{n-1}(E))_{\phi} \subset
\bX(N^n)$ under $\tilde{\Phi}$ maps to $\{0\}$ in $\mathrm{Ext}^{\sharp}_{A}(E,\hat{\mathbb{G}}_{a})$ (cf. Proposition $8.1$ in \cite{PS-1}). 
Hence this induces the map of short exact sequences of $R$-modules:
$${
\xymatrix{
0  \ar[r] &\dfrac{{\bX}_{n}(E)}{\bX(E)+\phi^{*}({\bX}_{n-1}(E))_{\phi}}\ar[d]^\Upsilon \ar[r] &{\bH}_{n}(E)\ar[d]^\Phi \ar[r] &\mathbf{I}_{n}(E)\ar@{^{(}->}[d] \ar[r]&0 \\
0 \ar[r] &\mathrm{Lie} (E)^{*} \ar[r] &\mathrm{Ext}^{\sharp}_{A}(E,\hat{\mathbb{G}}_{a}) \ar[r] &\mathrm{Ext}_{A}(E,\hat{\mathbb{G}}_{a}) \ar[r] &0.}
}$$
Finally, taking the direct limit with respect to $n$ gives us:
\label{map to_deRham}
\begin{equation} 
{\xymatrix{
0  \ar[r] &{\bX}_{\mathrm{prim}}(E)\ar[d]^\Upsilon \ar[r] &{\bH}(E)\ar[d]^\Phi \ar[r] &\mathbf{I}(E)\ar@{^{(}->}[d] \ar[r]&0 \\
0 \ar[r] &\mathrm{Lie} (E)^{*} \ar[r] &\bH^{*}_{\mathrm{dR}}({E}) \ar[r] &\mathrm{Ext}_{A}(E,\hat{\mathbb{G}}_{a}) \ar[r] &0.}
}.
\end{equation}

\section{The map $\Upsilon$}
In this section, we establish that the map $\Upsilon$ induces a $K$-linear isomorphism between $\bXp(E)_K$ and $\mathrm{Lie} (E)^{*}_K$. We need some preparation. We now recall the following results that plays a crucial role in proving Theorem \ref{thm-1-intro}.
\begin{theorem}[Proposition 8.4 in \cite{BS b}]\label{comfact} There exists a unique $A$-linear map $\Delta$ that makes the following diagram commute:
$$\xymatrix{
N^{n+1} \ar[rr]^-{\phi \circ i-i \circ \mathfrak{f}}  \ar[d]_u & & J^{n}E \\
N^1 \ar@{..>}[rru]_\Delta & &
}$$
\end{theorem}

Let $\mathbf{x}=\bx_0 = (x^{1}, x^{2}, \dots, x^{d})$ be the coordinate system on $E$ , and let $\left(\mathbf{x}_{0},\mathbf{x}_{1},\ldots, \mathbf{x}_{n}\right)$ be the induced Buium--Joyal coordinate functions on $J^{n}E$, where $\mathbf{x}_{i}=(x^{1}_{i}, x^{2}_{i} \dots x^{d}_{i})$ are the $i$-th Witt coordinates associated to $\bx$ as in Section \ref{ExplicitJnE}. 
For any delta character $\Theta \in \bX_n(E)$, consider the differential map $D\Theta : \Lie(J^n E) \map \Lie(\hG)\simeq \hG$ between the tangent spaces at the identity sections. With respect to the Buium--Joyal coordinates, the differential map is given by
\begin{equation}
D\Theta=(P_{0},\ldots,P_{n})
\end{equation}\label{Derivative_matrix}
where $P_j \in \mb{Mat}_{h\times 1} (R)$.
% Let $\Psi_{i}:= \left( \begin{smallmatrix} \Psi_{i1} \\ \vdots \\ \Psi_{ih}  \end{smallmatrix}\right) $ be an $R$-basis of $\bX(N^i)/u^{*}\bX(N^{i-1})$ for each $i\geq 1$.

\begin{proposition}[Proposition $6.3$ in \cite{BS b}]
\label{diff}
Let $\Theta$ be a character in $\bX_n(E)$.
\begin{enumerate}
        \item We have
                $$
                \iota^*\phi^*\Theta = \mfrak{f}^*(\iota^*\Theta)+ \gamma_{\Theta}. \Psi_1,
                $$
                where $\gamma_{\Theta}=\pi P_0$. \\
                
        \item For $n\geq 1$, we have
                $$
                \iota^*(\phi^{\circ n})^*\Theta= (\mfrak{f}^{n-1})^* \iota^*\phi^*\Theta.
                $$
\end{enumerate}
\end{proposition}

By Proposition \ref{diff} and Proposition $8.1 (a)$ of \cite{PS-1}, the $R$-linear map admits the following explicit description:
\begin{align*}
\Upsilon: \bXp(E) &\to \Lie(E)^*\\
\Theta &\mapsto \gamma_{\Theta}/\pi
\end{align*}
where $\gamma_{\Theta}$ is as in the Proposition above.

Let $\BB= \{\Theta_{1},\dots, \Theta_{h}\}$ denote a primitive basis for $\bX_\infty(E)$, and let $o_1,\dots, o_h$ be the respective orders of these delta characters. Throughout the remainder of this section, we abbreviate $\gamma_{\Theta_j}$ as $\gamma_j$.  
By Theorem \ref{comfact} and Proposition \ref{diff}, the morphism $\Delta : N^1 \map J^{r}E$ yields the pull-back of the primitive characters for all $j=1,\dots, h$:
$$\Delta^*\Theta_j:N^1E \map \hG$$
which are explicitly given by
\begin{align}
\label{commute}
\Delta^*\Theta_j = \gamma_j\cdot\Psi_1
\end{align}
where $\gamma_j = \pi P_{0j}$.

For any integer $n \geq m_{\mathrm{u}}$ and any character $f \in \bX_n(E)$, we can write
\begin{align}
f = \sum_{j=1}^h \left(\sum_{s=0}^{o_j} c_{sj}\phi^{\circ(n-s)*}\Theta_{j}
\right).
\end{align}
Applying Proposition \ref{diff}, we note that $\Delta^*(\phi^*\Theta)=0$ and we subsequently obtain
\begin{align}
\label{phiup}
\Delta^*f &
= (\phi(c_{01}) \gamma_{i_1} + \cdots + \phi(c_{0r})\gamma_{i_h})\cdot\Psi_1.
\end{align}

Let $\oE$ denote the special fiber of $E$ over the residue field $l$ of $R$, and let $\overline{F}$ be the geometric Frobenius on $\oE$. Suppose 
$$
P(X)= X^{r} + a_{r-1}X^{r-1} + \cdots + a_1 X + a_0
$$ 
is the characteristic polynomial in $A[X]$ such that $a_0\neq 0$ and $P(\overline{F})=0$ in $\End(\oE)$ (cf. Section 3.7 in \cite {Hu-Pa}). 
For any $\pi$-adically complete $R$-algebra $B$, we have the following commutative diagram:
\begin{equation}\label{frob-diag}
\xymatrix{
\bb{W}_{r}^h(B) \ar[r]^-\sim &J^{r}(\hG^h)(B) \ar[d]_-{P(\mfrak{f})} 
 & N^{r+1}(B) \ar[l]_{J^{r}(\Psi_1)}
\ar[r]^-{\phi\circ \iota} \ar[d]_-{P(\mfrak{f})} & J^{r}E(B) \ar[r] 
\ar@{.>}[dl]_{\mf g} \ar[d]_-{P(\mfrak{\phi})} & \oE(B\otimes k) 
\ar[d]_-{P(\overline{F})} \\
& \hG^h(B) \ar[d]_{\mathrm{pr_j}} & N^1(B) \ar[l]_{\Psi_1}
\ar[r]^-{\phi \circ \iota} & E(B) \ar[r]^-{\mf h} & \oE(B\otimes k)\\
& \hG (B) & & & 
}
\end{equation}
where $\mathrm{pr}_j: \hG^h \map \hG$ is the projection onto the 
$j$-th component for all $j=1,\dots,h$ and ${\mf h}: E(B) \map \oE(B/\pi B)$ 
is the natural group homomorphism
induced by the reduction mod $\pi$ given by $B \map B/\pi B$.

Observe that $\mathrm{Im(\phi \circ \iota)} = \mathrm{Ker}(\mf h)$. Since $P(\overline{F})=0$, the image of the morphism $P(\phi)$ lands inside $\ker {\mf h}$ and
hence induces the map
$\mf g:J^{r}E(B) \map N^1(B)$ as illustrated above.

\begin{theorem}
\label{isoupsilon}
The map $\Upsilon: \bXp(E) \map \mathrm{Lie} (E)^{*} $ is an injective morphism of $R$-modules with a $\pi$-torsion cokernel.

In particular, $\Upsilon_K: \bXp(E)_K \map \mathrm{Lie} (E)^{*}_K$ is an isomorphism of $K$-vector spaces.
\end{theorem}

\begin{proof}
It suffices to demonstrate that $\Upsilon_K$ is an isomorphism of $K$-vector spaces.
For all $j=1,\dots, h$, define the delta character of order $r$ as $\tTheta_j:J^{r}E \map \hG$ given by
\begin{align}
\tTheta_j:= \mathrm{pr}_j \circ \Psi_1 \circ \mf g.
\end{align}
From the preceding commutative diagram \eqref{frob-diag}, we have $(\phi \circ \iota)^*\mf g=P(\mff)$ and deduce 
\begin{align}
(\phi \circ \iota)^* \tTheta_j:= (\phi \circ \iota)^*(\mathrm{pr}_j \circ \Psi_1 \circ \mf g)= \pr_j \circ \Psi_1 \circ P(\fra).
\end{align}
Note that the morphism $\Psi_1 \circ P(\fra): N^{r+1}E \map \hG^h$ evaluates as
\begin{align}
\Psi_{r+1}+ a_{r-1}\Psi_{r} + \cdots + a_0 \Psi_1.
\end{align}
for some $a_i\in R.$
Consequently, we have
\beqar
(\phi \circ \iota)^* \tTheta_j &=& \pr_j \circ (\Psi_1 \circ P(\fra))\\
&=& \pr_j \circ (\Psi_{r+1}+ a_{r-1}\Psi_{r} + \cdots + a_0 \Psi_1 )\\
&=& \pr_j \circ (\Psi_{r+1}+ a_{r-1}\Psi_{r} + \cdots + a_1 \Psi_2) 
+ \pr_j \circ a_0 \Psi_1. \\
\eeqar
By Proposition \ref{diff} (1) we have 
\begin{align*}
\fra^* (\iota^* \tTheta_j) &= (\phi \circ \iota)^* \tTheta_j -\gamma_{\tTheta_j }\\
&=\pr_j \circ (\Psi_{r+1}+ a_{r-1}\Psi_{r} + \cdots + a_1 \Psi_2)+(a_0 e_j-\gamma_{\tTheta_j })\cdot \Psi_1.
\end{align*}
where $e_j$ represents the $j$-th row of the $h\times h$ identity matrix $\mathbbm{1}_{h}$.
On the other hand, by Lemma \ref{ind-psi}, we know $\{\Psi_1,\ldots,\Psi_r\}$ forms a $K$-basis of $\Hom_A(N^r,\hG^{h}).$ This in particular implies: for some $b_i\in K,$ we have
\begin{align*}
\fra^* (\iota^* \tTheta_j)&=\mff^*(b_r\Psi^r+\ldots+b_1\Psi_1)\\
&=\phi(b_r)\Psi_{r+1}+\ldots+\phi(b_1)\Psi_2
\end{align*}
Therefore, comparing both the representations of $\fra^* (\iota^* \tTheta_j)$, we obtain $( a_0 e_j-\gamma_{\tTheta_j })\cdot \Psi_1=0$ and
\begin{align}
\label{phi-up2}
\Delta^* \tTheta_j := (\phi \circ \iota - \iota \circ \fra)^*\tTheta_j =
 a_0 e_j\cdot \Psi_1
\end{align}

Whereas, according to (\ref{phiup}), for each $j$ there exist coefficients $c_{1j},\ldots, c_{hj} \in R$ such that
\begin{align}
\label{phi-up3}
\Delta^* \tTheta_j &= (\phi(c_{1j}) \gamma_1 + \ldots + 
\phi(c_{hj}) \gamma_h)\cdot \Psi_1.
\end{align}
By equating expressions (\ref{phi-up2}) and (\ref{phi-up3}) for all $j$, we deduce that 
\begin{align}
\label{phi-up4}
\left(\begin{smallmatrix} \gamma_1  \\ \vdots \\  \gamma_h \end{smallmatrix}\right)\cdot C = a_0 \mathbbm{1}_h
\end{align}
where $C= (\phi(c_{ij})) \in \Mat_{h\times h}(R)$. Since $a_0$ is the non-zero constant coefficient of the characteristic polynomial of Frobenius on the special fiber of $E,$ the matrix $\left(\begin{smallmatrix} \gamma_1  \\ \vdots \\  \gamma_h \end{smallmatrix}\right)$ is invertible in $K.$

This implies the map $\Upsilon$ to be invertible in $K$, and we conclude that $\Upsilon:\bXp(E)_K \simeq \mathrm{Lie} (E)^{*}_K$.
\end{proof}

\section{The $z$-isocrystal $\bH_{\d}(E)$ and its Hodge--Pink structure}
\label{LocalShtuka}
In this section, we associate to any Anderson module $E$ a canonical $z$-isocrystal equipped with a Hodge--Pink structure, denoted by $(\wHH, \qH,\fra^*)$, and in particular prove the non-degeneracy of $\mff^*.$ 

In the equal characteristic setting, the $z$-isocrystals with Hodge--Pink structures introduced by Pink serve as the natural analogs of filtered isocrystals in $p$-adic Hodge theory. 
Let $l$ denote the residue field of $R$, and let $K$ denote the fraction field of $R$. We denote the lift of the Frobenius endomorphism $\sigma$ on $l[[z]]$ by the assignments $z\mapsto z$ and $a\mapsto a^q$ for any $a\in l$. Additionally, consider the ring homomorphism $l((z)) \map K[[z - \pi]]$ defined by $z \mapsto \pi + (z - \pi)$.

\begin{definition} A \textit{$z$-isocrystal} over $l$ is a pair $(D,F_{D})$ such that:
\begin{itemize}
\item[(i)] $D$ is a finite-dimensional vector space over $l((z))$;
\item[(ii)]  $\sigma^{*}D=D\otimes_{\sigma,~l((z))}l((z))$, and $F_{D}:\sigma^{*}D\longrightarrow D$ is an isomorphism of $l((z))$-vector spaces.
\end{itemize}
\end{definition}

Suppose $E$ is an Anderson $A$-module of dimension $d$ and rank $r$ over $R$. By Theorem \ref{finite_dim_thm}, $\HH(E)$ is a finite free module over $R\simeq l[[z]]$. Under this identification, the lift of the Frobenius endomorphism $\phi:l[[z]]\longrightarrow l[[z]]$ acts by fixing $z$ and sending $a\mapsto a^{q}$ for all $a\in l$.
We define the $l((z))$-vector space $\wHH:=\bfH\otimes l((z))$. 
The semilinear operator $\fra^*$ naturally extends to $\wHH$.
It follows immediately from the definitions that $(\wHH,\mathfrak{f}^{*})$ constitutes a $z$-isocrystal provided that $\det(\mathfrak{f}^{*})$ is nonzero; this non-degeneracy condition will be rigorously established in the subsequent subsection.

\subsection{Non-degeneration of the $z$-isocrystal $\bH_{\d}(E)$}
As an essential consequence of Theorem \ref{isoupsilon}, we will show that the semilinear operator $\fra^*$ on $\bH(E)_K$  is a bijective map—which confirms that $\bH_{\d}(E)$ is indeed a $z$-isocrystal.

Let $\{\Theta_1,\dots, \Theta_h\}$ be primitive characters such that $o_i = \ord{\Theta_i} \leq \mup$ for all $i=1,\dots , h$. The existence of such a basis is guaranteed by Theorem 7.6 of \cite{PS-1}. The module $\bXp(E)_K$ is then generated by the images of these $\Theta_i$ over $K$. 

For any delta character $\Theta$, define the set:
$$
S_n(\Theta):= \{\phi^{*(n-i)}\Theta \mid i = \ord{\Theta},(\ord{\Theta}+1), 
\dots , n\}
$$
for all $n \geq \ord{\Theta}$. Recall from page 418 of \cite{BS b} that for all $n\geq \mup$, the sets $S_n(\Theta_1),\dots, S_n(\Theta_h)$ are $K$-linearly independent and freely generate $\bX_n(E)_K$ as a $K$-module.

Define the delta characters $\tTheta_i:= \phi^{*(n-o_i)}\Theta_i$ for all $n\geq \mup$ and $i=1,\dots, h$.
Consequently, the quotient $K$-module $\bX_n(E)_K/u^*\bX_{n-1}(E)_K$ is generated by the set $\{\tTheta_1,\dots ,\tTheta_h\}$.

Suppose for all $i=1,\dots, h$ we can write
\begin{align}
\iota^* \tTheta_i = 
\lam_{i{\mup}}\Psi_{\mup} - \lam_{i({\mup} -1)}\Psi_{{\mup} -1} 
- \cdots - \lam_{i1} \Psi_1
\end{align}
where $\lam_{ij} \in \Mat_{1\times h}(R)$ for all $j = 1,\dots \mup$.
This yields 
\begin{align}
\fra^* \iota^* \tTheta_i = 
\lam_{i{\mup}}^\phi\Psi_{{\mup}+1} - \lam_{i({\mup} -1)}^\phi\Psi_{{\mup}} 
- \cdots - \lam_{i1}^\phi \Psi_2.
\end{align}
By Proposition $\ref{diff}$, there exist matrices $\gamma_i \in \Mat_{1\times h}(R)$ such that
\begin{align}
\label{phimu}
\nonumber
\iota^*\phi^* \tTheta_i &= \fra^*\iota^* \tTheta_i + \gamma_i \Psi_1\\
&= 
\lam_{i{\mup}}^\phi\Psi_{{\mup}+1} - \lam_{i({\mup} -1)}^\phi\Psi_{{\mup}} 
- \cdots - \lam_{i1}^\phi \Psi_2 + \gamma_i\Psi_1.
\end{align}
For each $j=1,\dots ,\mup$, define the $h \times h$ matrices 
$$\Lambda_{j}:= \left(\begin{matrix}
\lam_{1j}^\phi \\ \vdots \\ \lam_{hj}^\phi \\
\end{matrix}
\right)$$
and 
$$\Gamma:= \left(\begin{matrix}
\gamma_1 \\ \vdots \\ \gamma_h \\
\end{matrix}
\right).$$
Therefore, the induced map
$$
[\iota^*\phi^*]:\bX_n(E)_K/u^*\bX_{n-1}(E)_K {\longrightarrow}
\bX(N^{n+1})_K/u^*\bX(N^n)_K$$
is completely determined by
$[\iota^*\phi^*] \tTheta_i = \lam^\phi_{i\mup} \Psi_{{\mup}+1}$ for all $i=1,\dots, h$. In terms of the basis $\{\tTheta_1,\dots,\tTheta_h\}$ of $\bX_{\mup}(E)/u^*\bX_{{\mup}-1}(E)$ and the basis $\{\Psi_{\mup+1}\}= \{\Psi_{(\mup+1)1},\dots , \Psi_{(\mup+1)h}\}$ of 

\noindent
$\bX(N^{\mup +1})/u^*\bX(N^{\mup})$, the map $[\iota^*\phi^*]$ is represented by:
\begin{align}
[\iota^*\phi^*] = \Lambda_{\mup}.
\end{align}

%Recall the following result
%\begin{theorem}[Proposition 8.6 in \cite{BS b}.]\label{frob_bij}
 %$${
%\xymatrix{
%{\bX}_{n}(E)/{\bX}_{n-1}(E)\ar@{^{(}->}[d]_{i^{*}} \ar@{^{(}->}[r]^{\phi^{*}} &{\bX}_{n+1}(E)/{\bX}_{n}(E)\ar@{^{(}->}[d]^{i^{*}}  \\
%{\bX}(N^{n})/{\bX}(N^{n-1}) \ar[r]^{\mathfrak{f}^{*}}_{\sim} &{\bX}(N^{n+1})/{\bX}(N^{n}) }}
%$$
%where $i^{*}$ and $\phi^{*}$ are injective, and $\mathfrak{f}^{*}$ is bijective. 
%\end{theorem}
The following proposition can be proven similarly to Proposition $8.1$ in \cite{BS a}.
\begin{proposition}
\label{Lambdamu}
For all $n \geq \mup$, the natural map
$$
\bX_n(E)_K/u^*\bX_{n-1}(E)_K \stk{[\iota^* \phi^*]}{\longrightarrow}
\bX(N^{n+1}E)_K/u^*\bX(N^nE)_K
$$
is an isomorphism. 

In particular, for $n=\mup$, the map 
$$
\bX_{\mup} (E)_K/u^*\bX_{\mup-1}(E)_K \stk{[\iota^* \phi^*]}{\longrightarrow}
\bX(N^{{\mup} +1})_K/u^*\bX(N^\mup)_K
$$
is an isomorphism, which implies that $\Lambda_{\mup}$ is an invertible matrix over $K$.
\end{proposition}

\subsection{The operator $\fra^*$}
Let us review the explicit description of the operator $\fra^*$ on $\bH(E)_K$.
Recall that 
$$
\bH(E)_K = \varinjlim_{n} 
\frac{\bX(N^n)_K}{\iota^* \phi^* \bX_{n-1}(E)_K}
$$
and that $\fra^*$ on $\bH(E)_K$ is induced from the pullback on the delta characters of $N^nG$, given by $\fra^* \Psi_i = \Psi_{i+1}$ for all $i \geq 1$.
Because $\fra^*(\iota^* \phi^* \bX_{n-1}(E)) \subset \iota^* \phi^* \bX_n(E)$ for all $n\geq 1$, $\fra^*$ cleanly descends to a semilinear operator on $\bH(E)_K$.

By equation (\ref{phimu}), we have 
\begin{align}
\fra^* \Psi_{\mup} = \Psi_{\mup+1} \equiv
\Lambda_{\mup}^{-1}\left( 
\Lambda_{\mup-1}\left(\Psi_{\mup} + \cdots + \Lambda_1 \Psi_2 - \Gamma \Psi_1\right)\right) \bmod
\iota^*\phi^* \bX_{\mup-1}(E)_K
.\nonumber
\end{align}
Employing the same structural arguments as in Proposition $8.1$ of \cite{BS a}, we establish that for all $n \geq \mup$: 
$$
\bH(E)_K \simeq \bH_n(E)_K \simeq \bH_{\mup}(E)_K.
$$
where $\bH_{\mup}(E)_K \simeq K\langle\Psi_1,\dots , \Psi_{\mup}\rangle/ \iota^*\phi^* \bX_{\mup -1}(E)$.

Consider the associated operator (still denoted by $\fra^*$) on $\Hom(N^{\mup},\hG)_K \simeq K\langle \Psi_1,\dots , \Psi_{\mup}\rangle$, defined explicitly as:
\begin{align}
\fra^*(\Psi_i) = \Psi_{i+1} \mb{ for all } i= 1, \dots, \mup -1, \mb{ and }
\nonumber \\
\fra^* \Psi_{\mup} = \Lambda_{\mup}^{-1} \left({\Lambda_{\mup-1}} 
\Psi_{\mup} + \cdots + \Lambda_1 \Psi_2 - \Gamma \Psi_1\right).\nonumber
\end{align}
This explicit map precisely descends to $\fra^*$ on the quotient $\bH(E)_K$.
With respect to the block basis $\{\Psi_1, \dots, \Psi_{\mup}\}$, the $(h\mup  \times h\mup)$-matrix representing $\fra^*$ is given by: 
\begin{align}
[\fra^*] = 
\left(
\begin{matrix}
[0]_h & [0]_h & \cdots & & - \Lambda_{\mup}^{-1}\Gamma \\
[\mathbbm{1}]_h & [0]_h &  & & \Lambda_{\mup}^{-1}\Lambda_1 \\
[0]_h & [\mathbbm{1}]_h &  & &   \\
\vdots & & \ddots & & \vdots \\
[0]_h & [0]_h & \cdots  &[\mathbbm{1}]_h & 
\Lambda_{\mup}^{-1}\Lambda_{\mup-1} \\
\end{matrix}
\right)
\end{align}
where $[0]_h$ and $[\mathbbm{1}]_h$ denote the $h \times h$ zero matrix and identity matrix, respectively. 

\begin{theorem}
\label{bijfra}
The semilinear operator $\fra^*: \bH(E)_K \map \bH(E)_K$ is a bijection.
\end{theorem}
\begin{proof}
By equation (\ref{phi-up4}) in the proof of Theorem \ref{isoupsilon}, we know that the block matrix $\Gamma = (\gamma_1 \cdots \gamma_h)^t$ is an invertible $h\times h$ matrix over $K$.
Consequently, taking the determinant of the block companion matrix yields:
\begin{align}\label{det-fra}
\det [\fra^*] = (-1)^{\mup} \det (\Lambda_{\mup}^{-1} \Gamma) \ne 0.
\end{align}
This non-vanishing determinant implies that the semilinear map $\fra^*: \bX(N^{\mup})_K \map \bX(N^{\mup})_K$ is bijective. Hereby the operator $\fra^*$ on the quotient space $\bH(E)_K$ is also bijective.
\end{proof}

\subsection{Hodge-Pink structure on $\bH_{\d}(E)$}\label{HP_delta}
\begin{definition} A \textit{Hodge--Pink structure} over $K$ on a $z$-isocrystal $(D,F_{D})$ is a $K[[z-\pi]]$-lattice $\mathfrak{q}_{D}$ contained within $\sigma^{*}D\otimes_{l((z))}K((z-\pi))$.
\end{definition}

A triple $(D',\mathfrak{q}_{D'},F_{D'})$ is said to be a \textit{sub-object} of $(D,\mathfrak{q}_{D},F_{D})$ if $D'$ is a subspace of $D$, $F_{D'}=F_{D}|_{D'}$, and $\mathfrak{q}_{D'}\subset \mathfrak{q}_{D}\cap (\sigma^{*}D'\otimes K((z-\pi)))$. Moreover, it is termed \textit{strict} if $\mathfrak{q}_{D'}= \mathfrak{q}_{D}\cap (\sigma^{*}D'\otimes K((z-\pi)))$.
Given any $z$-isocrystal $(D,F_{D})$, one can canonically associate a tautological Hodge--Pink structure $\mathfrak{p}_{D}=\sigma^{*}D\otimes K[[z-\pi]]$. For all integers $i$, the Hodge--Pink structure then induces a filtration on $D_K$ as follows:
\begin{align}
\Fil^{i}_{D_{K}}:=\mathrm{Im}\left(\mathfrak{p}_{D}\cap (z-\pi)^{i}\mathfrak{q}_{D}\right)~ \mathrm{in}~ D_{K}.
\end{align}
We will enhance the $z$-isocrystal $(\wHH,\mff^*)$ by associating a Hodge-Pink structure on it as follows:

%Note that since $\Lambda_{m_u}$ is invertible over $R$, by \eqref{det-fra}, we get  $\mathrm{ord}_{z}\left(\mathrm{det}(\mathfrak{f}^{*})\right)\geq \mathrm{ord}_{z}\left(\mathrm{det}(\Gamma)\right)\geq1.$ 
Let $e=\mathrm{ord}_{z}\left(\mathrm{det}(\mathfrak{f}^{*})\right)$. 
Recall that $\XX_{\mathrm{prim}}(E)=\langle \Theta_{1},\ldots,\Theta_{h}
\rangle$ where $\Theta_1,\dots, \Theta_h$ are primitive characters of $E$. 
Extending this basis to $\wHH$, we get $\wHH=\langle i^{*}\Theta_{1},\ldots,i^{*}\Theta_{h},\Psi_{i_1},\dots , \Psi_{i_m}\rangle$, for some $\Psi_{i_j}\in \bfH$. We define the Hodge-Pink structure on $\wHH$ as the following submodule
of $\phi^*\bH(E) \otimes K((z-\pi))$:
\begin{align*}
	\qH:=\phi^{*}\langle(z-\pi)^{-e}i^{*}\Theta_{1},\ldots,(z-\pi)^{-e}i^{*}\Theta_{h},\Psi_{i_1},\dots \Psi_{i_m}\rangle\otimes_{l((z))} K[[z-\pi]].
\end{align*}
An easy computation shows that the filtration is given by:
\begin{align*}
\Fil^{i}_{\wHH}=\left\{\begin{array}{ll}
		{\HH(E)}_{K}, & \mb{if } i\leq 0 \\
		{\XX_{\mathrm{prim}}(E)}_{K}, & \mb{if } 1\leq i\leq e\\
		0, & \mb{if} ~i\geq e+1.
	\end{array} 
	\right.
\end{align*}
By Theorem \ref{bijfra}, we have the following conclusion:
\begin{corollary} The object $(\wHH, \qH,\fra^*)$ is a $z$-isocrystal with Hodge--Pink structure.
\end{corollary}
\section{Weak admissibility of $\bH_{\d}(E)$ for Drinfeld modules and Galois Representations}\label{weak-admissible}
Let $(D,\mathfrak{q}_{D},F_{D})$ be a $z$-isocrystal with a Hodge-Pink structure. We can associate two invariants to $(D,\mathfrak{q}_{D},F_{D})$: the Newton number $t_{N}(D)$ and the Hodge number $t_{H}(D)$. 
The Newton number for $(D,F_D)$ is defined as
\begin{align}
	t_{N}(D):=\ord_z(\mathrm{det}(F_{D}))
\end{align}
and the Hodge number as 
\begin{align}
t_{H}(D):=\sum_{i\geq 0}i.\mathrm{dim}\left(\dfrac{\Fil^{i}_{D_{K}}}
{\Fil^{i+1}_{D_{K}}}\right).
\end{align}
\begin{definition}
	A $z$-isocrystal $(D,F_D)$ with Hodge Pink structure is said to be weakly admissible if 
	\begin{itemize}
		\item[(i)] $t_{H}(D)=t_{N}(D)$ and
		\item[(ii)] $t_{H}(D')\leq t_{N}(D')$ for any strict sub-object $D'$ of $D$.
	\end{itemize}
\end{definition}
Here we now recall some of the basic facts from \cite{BS b} when $E$ is a 
Drinfeld module of rank $r$. For a Drinfeld module $E$ of rank $r$ let $m$
be the splitting number. Then $\bXp(E) \simeq R\langle \Theta \rangle$ where 
$\Theta$ is a primitive character of order $m$ satisfying
\begin{align}
\label{Thetapull}
i^*\Theta = \Psi_m - \lam_{m-1}\Psi_{m-1} - \cdots - \lam_1 \Psi_1.
\end{align}
As in Section 9B of \cite{BS b}, we have
$\HH(E) \simeq R\langle \Psi_1,\dots \Psi_m \rangle$. 
The infinite rank free $R$-module 
$$\HomA(N^\infty,\hG) := \varinjlim~ (\HomA(N^n,\hG)) \simeq 
R\langle \Psi_1, \Psi_2, \dots \rangle$$ has semilinear operator $\fra^*$ 
induced from pulling back via the lateral Frobenius $\fra$ on $N^n$s.
This is given by 
$$
\fra^*(\Psi_i) = \Psi_{i+1}
$$
for all $i \geq 1$. Then the induced semilinear operator $\fra^*$ acts on 
the quotient $R$-module $\HH(E)$ as 
\begin{align*}
\fra^*(\Psi_i) &= \Psi_{i+1}, \mb{ for all } i = 0,\dots , m-1\\
\fra^*(\Psi_m) &= \phi(\lam_{m-1})\Psi_m + \cdots + \phi(\lam_1)\Psi_1 + 
\gamma \Psi_1.
\end{align*}
However by $(\ref{Thetapull})$, we can consider $\HH(E)$ with the new basis
$R\langle i^*\Theta, \Psi_1,\dots ,\Psi_{m-1}\rangle$. 
Then by Proposition \ref{diff}  we get
\begin{align}
\label{franew}
\fra^*(i^*\Theta) &= \gamma \Psi_{1} \\
\fra^*(\Psi_1) &= \Psi_2 \nonumber\\
\vdots \hspace{.75cm} & \hspace{.75cm} \vdots \nonumber \\
\fra^*(\Psi_{m-2}) &= \Psi_{m-1} \nonumber  \\
\fra^*(\Psi_{m-1}) &= i^*\Theta + \lam_{m-1}\Psi_{m-1}+ \cdots + \lam_1 \Psi_1
\nonumber
\end{align}
Let $\overline{L}$ denote a fixed algebraic closure of $l((z))$. For $r$ positive, set $L_{r}= l((z^{\frac{1}{r}}))\subset \overline{L}$ and $L_{\infty}:=l((z^{1/\infty}))=\displaystyle{\bigcup_{r\geq 1}}L_{r}$. Let $\mathcal{O}_{L_{\infty}}$ denotes the valuation ring of $L_{\infty}$. Then we can extend $\phi$ uniquely to $L_{\infty}$ as $\phi(z^{\frac{1}{r}}) = z^{\frac{1}{r}}$, for each $r\geq 1$.
\begin{lemma}
\label{eqval}
For any $x \in \mathcal{O}_{L_{\infty}}$, we have $\ord_z(x) = \ord_z(\phi(x))$.
\end{lemma}
\begin{proof}
For $x\in \mathcal{O}_{L_{\infty}}$ we have  $\phi(x) = x^\hq+ z \d x$. Let $e:= \ord_z(x)$. If $e=0$,
then $\ord_z(x^\hq) = 0$ which implies $\ord_z(\phi(x)) = 0$. 
Now when $e > 0$, then $x = z^eu$ where $\ord_z(u) =0$. Then $\phi(x) = z^e
\phi(u)$. From $e=0$ case, we have $\ord_z(\phi(u))=0$ and hence we 
have $\ord_z(\phi(x)) = e$ and we are done.
\end{proof}

\begin{lemma}
\label{shortlem}
Let $a \beta = b + a c$ where $\ord_z(a)< 0$ and $\ord_z(b), \ord_z(c) \geq 0$.
Then $\ord_z(\beta) \geq 0$.
\end{lemma}
\begin{proof}
Let $\ord_z(a) = -h$ where $h > 0$. Then 
$$\ord_z(a \beta) = \ord_z(a) + \ord_z(\beta) = -h+ \ord_z(\beta).$$ 
Also note that 
$$\ord_z(b+ac) \geq \min(\ord_z(b), \ord_z(ac)).$$
Hence combining the above we obtain 
$-h + \ord_z(\beta) \geq \left\{\begin{array}{l}\ord_z(b) \mb{, or}\\ 
\ord_z(ac). \end{array}\right.$
Then the first case of the inequality implies 
\begin{align*}
-h + \ord_z(\beta) &\geq \ord_z(b) \geq  0 \\
\ord_z(\beta) &\geq h \geq 0.
\end{align*}
On the other hand, the second case implies
\begin{align*}
-h+ \ord_z(\beta) &\geq -h + \ord_z(c)\\
\ord_z(\beta) &\geq \ord_z(c) \geq 0.
\end{align*}
and this completes our proof.
\end{proof}

\begin{lemma} 
\label{eigenint}
Let $E$ be a Drinfeld module of rank $r$. Consider $\wHH_{L_{\infty}}$ as an
$L_{\infty}$-vector space. Let $0\neq v\in\wHH_{L_{\infty}}$ be such 
that $\mathfrak{f}^{*}(v)=\mu v$. Then $\ord_z(\mu) \geq 0$.
\end{lemma}
\begin{proof}
Let $m$ be the splitting number of $E$ and $v \in \wHH_{L_{\infty}}$ be a non-zero 
eigenvector with eigenvalue $\mu \in L_{\infty}$. In terms of the basis given above, 
write $v=\beta_{0}i^{*}\Theta+\beta_{1}\Psi_{1}+\ldots+\beta_{m-1}\Psi_{m-1}$, where $\beta_{i}\in {L_{\infty}}$. Then by (\ref{franew}) we have 
\begin{align}
\label{fstar} 
\fra^*v &= \phi(\beta_{m-1}) i^*\Theta + 
(\phi(\beta_0)\gamma+ \phi(\beta_{m-1}) \lam_1)\Psi + (\phi(\beta_1)+
\phi(\beta_{m-1})\lam_2)\Psi_2 \\
& \hspace{3cm} + \cdots + (\phi(\beta_{m-2}) + \phi(\beta_{m-1})\lam_{m-1})
\Psi_{m-1} \nonumber
\end{align}
Comparing the coefficients of the basis 
elements in the equation $\mathfrak{f}^{*}(v)=\mu v$, we obtain
\begin{align*}
\mu \beta_0 &=	\phi(\beta_{m-1})\\
\mu\beta_{1} &=	\phi(\beta_{0})\gamma+\lambda_{1}\phi(\beta_{m-1})\\
\mu\beta_{i} &= \phi(\beta_{i-1})+\lambda_{i}\phi(\beta_{m-1}), 
~\mathrm{for}~ 2\leq i\leq m-1.
	\end{align*}
Recall that $\gamma$ and $\lam_i$s are integral by Propositions $9.6$
 and $9.8$ in \cite{BS b}. Hence $\ord_z(\gamma), \ord_z(\lam_i) \geq 0$ for
all $i$.
If $\mu$ is zero then we are done. Hence we assume $\mu$ to be nonzero. Since $\phi$ is an automorphism, from above relations it is easy to see that 
$\beta_{0}$ equal to zero implies all $\beta_{i}$s are zero. Hence $v$ is zero. Therefore 
$\beta_{0}\neq 0$. Note that $\mathfrak{f}^{*}(\beta_{0}^{-1}v)=
\phi(\beta_{0}^{-1})\mu v=(\phi(\beta_{0}^{-1})\beta_{0}\mu)(\beta_{0}^{-1}v)$,
that is $\beta_0^{-1} v$ is also an eigen vector with eigen value 
$\phi(\beta_0^{-1})\beta_0 \mu$ whose $z$-adic valuation is the same as that of 
$\mu$ by Lemma \ref{eqval}. 
Hence without loss of generality, we can assume that $\beta_{0}$ is $1$. 
Hence substituting $\mu = \phi(\beta_{m-1})$ in the rest of the above 
equations we obtain
\begin{align*}
\phi(\beta_{m-1})\beta_1 &=\gamma+\lam_{1}\phi(\beta_{m-1})\\
\phi(\beta_{m-1})\beta_i&=\phi(\beta_{i-1})+\lambda_{i}\phi(\beta_{m-1}),
 ~\mathrm{for}~ 2\leq i\leq m-1.
\end{align*}
If $\ord_{z}(\beta_{m-1})\geq 0$, then from the first equation and
Lemma \ref{eqval} we have $\ord_{z}(\mu)\geq 0$ and we are done.  
Now assume $\ord_{z}(\beta_{m-1})<0$. Then applying Lemma \ref{shortlem}
inductively to the above set of equations, starting with the first one, 
we obtain $\ord_z(\phi(\beta_i)) \geq 0$ for all $i=1,\dots, m-1$. Then 
Lemma \ref{eqval} implies $\ord_z(\beta_i) \geq 0$ for all $i=1,\dots, m-1$.  
Hence in particular $\ord_z(\beta_{m-1}) \geq 0$ which contradicts our 
assumption that $\ord_z(\beta_{m-1}) < 0$. Therefore, by the first equation,
 we must have $\ord_z(\mu) \geq 0$ and this completes the proof.
\end{proof}

\begin{lemma}
\label{det_pos} 
Let $D\subset \wHH$ be a subspace stable under the semilinear 
operator $\mathfrak{f}^{*}$ and denote $F_D = \fra^*|_D$. 
Then $\ord_z(\mathrm{det(F_{D})})\geq 0$.
\end{lemma}
\begin{proof}

Note that without loss of generality, we may assume that $D$ is a vector space over $L_{\infty}$. Now we know from the Dieudonn\'{e} 
theory (cf. Proposition B.1.14, Appendix B in \cite{Lau}) that $D$ can be 
written as the direct sum of simple modules of the form 
$\frac{L_{\infty}\{\mathfrak{f}^{*}\}}{L_{\infty}\left({\mathfrak{f}^{*}}^{n}-z^{s}\right)}$, for some $n\geq 1$ and $k$ in $\ZZ$. The matrix with respect to the basis $\{1,\mathfrak{f}^{*},\ldots,{\mathfrak{f}^{*}}^{r-1}\}$ is given by 
$$\begin{bmatrix}
    0 & 0  & \dots & 0 & z^{s} \\
    1 & 0  & \dots & 0 & 0 \\
    0 & 1  & \dots & 0 & 0\\
   \vdots & \ddots & \ddots  & \vdots & \vdots\\
    0 & 0  & \dots & 1 & 0 & 
\end{bmatrix}$$
To prove our result it is enough to show that $s$ is non-negative for each
of the the simple modules appearing in the direct sum decomposition of $D$. 
Now note that a subobject $D \inj \wHH$, has a section in the category of 
$L_{\infty}\{\fra^*\}$-modules. Hence the simple modules appearing in the direct sum
decomposition of $D$ also appear in the decomposition of $\wHH$.
Hence without loss of generality, we can assume $D$ itself is a simple $L_{\infty}\{\fra^*
\}$-module of the above form. Consider the element 
$$v=\sum_{i=0}^{n-1} z^{\frac{(n-i-1)s}{n}}{\mathfrak{f}^{*}}^{i}\in D.$$
Then $v$ is an eigenvector of $\fra^*$ with eigenvalue $z^{\frac{s}{n}}$ that 
is $\mathfrak{f}^{*}(v)=z^{s/n}v $. Then the above Lemma \ref{eigenint} 
implies that $s\geq 0$.
\end{proof}

\begin{theorem}
\label{wadd}
Let $E$ be a Drinfeld module of rank $r$ over $R$.
\begin{itemize}
\item[(1)]  Then the $z$-isocrystal with Hodge-Pink structure  $(\wHH,\qH,\mathfrak{f}^{*})$ is weakly admissible.	

\item[(2)]  In particular, $(\wHH, \qH, \fra^*)$ determines a crystalline Galois representation
$$
\rho_{\wHH}: \mathrm{Gal}(K^{\mathrm{sep}}/K) \longrightarrow 
\mathrm{GL}_m(\FF_{\hq}((z)))
$$
where $m = \dim_{l((z))} \wHH$ is the splitting number of $E$.
\end{itemize}
\end{theorem}

\begin{proof}
$(i) $
Let $E$ be a Drinfeld module of rank $r$ with splitting number $m$. 
Then by Section 9B of \cite{BS b} and expressing $\Psi_m$ as in 
(\ref{Thetapull}) we 
obtain $\bfH=\langle i^{*}\Theta,\Psi_{1},\dots,\Psi_{m-1}\rangle$. Hence the tautological Hodge-Pink structure is given by $\mathfrak{p}_{\wHH}:=\phi^{*}\langle i^{*}\Theta,\Psi_{1},\dots,\Psi_{m-1}\rangle$. Also by our definition $$\qH:=\phi^{*}\langle(z-\pi)^{-e}i^{*}\Theta,\Psi_{1},\dots,\Psi_{m-1}\rangle\otimes_{l((z))} K[[z-\pi]]\subset \phi^{*}\bfH\otimes K((z-\pi))$$
where $e=\ord_{z}(\gamma)=\ord_{z}(\mathrm{det}(\mathfrak{f}^{*}))$. Recall that the induced filtration is given by:
\begin{align*}
\Fil^{i}_{\wHH}=\left\{\begin{array}{ll}
		{\HH(E)}_{K}, & \mb{if } i\leq 0 \\
		{\XX_{\mathrm{prim}}(E)}_{K}, & \mb{if } 1\leq i\leq e\\
		0, & \mb{if} ~i\geq e+1.
	\end{array} 
	\right.
\end{align*}
Therefore $t_{H}(\wHH)=e=t_{N}(\wHH)$. Let  $(D',\mathfrak{q}_{D'},F_{D'})$ be a strict sub-object of $(\wHH,\qH,\mathfrak{f}^{*})$. Then 
$$\mathfrak{q}_{D'}= \qH\cap \phi^{*}D'\otimes K((z-\pi))$$ 
Note that $\fra^*(i^*\Theta) = \gamma \Psi_1$. Since
$\gamma$ is nonzero, we have $D'\neq \langle i^{*}\Theta \rangle$. Therefore $\mathfrak{q}_{D'}=\phi^{*}D'\otimes K[[z-\pi]]=\mathfrak{p}_{D'}$. Hence the induced filtration becomes 
\begin{align*}
	\Fil^{i}_{D'}=\left\{\begin{array}{ll}
		{D'}_{K}, & \mb{if } i\leq 0 \\
		0, & \mb{if} ~i\geq 1.
	\end{array} 
	\right.
\end{align*}
Hence we have $t_{H}(D')=0$. On the other hand, by Lemma \ref{det_pos}, we already have $$t_{N}(D')=\ord_{z}(\mathrm{det}(F_{D'}))\geq 0.$$ 
and this proves $(i)$.

$(ii)~$
This follows from $(i)$ along with 
the theory of local shtukas developed by Geneister and 
Lafforgue, Hartl that associates crystalline Galois representations
$\rho: \mathrm{Gal}(K^{\mathrm{sep}}/K) \map \mathrm{GL}_m(\FF_{\hq}((z)))$ to 
local shtukas of rank $m$ over $R$ with quasi-morphisms. The essential 
image of the $\bb{H}$ functor, c.f. \cite{Hartl 2011} page $1245$, 
inside the category of 
$z$-isocrystals with Hodge-Pink structure are the admissible ones. By 
Theorem 2.5.3 of \cite{Hartl 2011}, all weakly admissible $z$-isocrystals 
with Hodge-Pink structure over $K$ are admissible. Then $(i)$ implies $(\wHH,\qH, \fra^*)$ is admissible and therefore 
associates a crystalline Galois representation $\rho_{\wHH}:\mathrm{Gal}
(K^{\mathrm{sep}}/K) \map \mathrm{GL}_m(\FF_{\hq}((z)))$.
\end{proof}

\section{The cases of splitting numbers $m=1$ and $m=r$}\label{comparison}

\subsection{Canonical Lift (CL) Drinfeld modules}
A Drinfeld module $E$ is defined to be a \emph{canonical lift} (CL) if it admits an $A$-linear endomorphism $\psi \in \mathrm{Hom}_A(E,E)$ satisfying $\psi(x) \equiv x^\hq \pmod{\pi}$ for all $x \in \hat{\mathbb{G}}_{a}$. 

By \cite[Theorem 8.3]{BS b}, the splitting number $m$ of $E$ is equal to $1$ if and only if $E$ is a canonical lift. In this regime, the underlying module takes the form $\bH(E)=R\langle\Psi_{1}\rangle$, and \cite[Lemma 9.3]{PS-1} dictates that $\gamma=-\pi$. Furthermore, by Proposition \ref{diff}(i), the semilinear operator $\mff^{*}$ acts via 
$$ \mff^{*}(\Psi_{1}) = \mff^{*}(i^{*}\Theta_{1}) = -\gamma \Psi_{1} = \pi \Psi_{1}. $$ 
Consequently, the associated Galois representation is strictly one-dimensional.
 
\subsection{Carlitz Modules}\label{rk-1-com}
Let $E$ be a Carlitz module, defined classically as a Drinfeld $A$-module of rank $1$. Its $A$-module structure is explicitly governed by
$$ \varphi_{E}(z) = \pi z + z^{q}. $$ 
The module $E$ naturally admits a canonical lift via $\varphi_{E}(z)$. By \cite[Lemma 9.3]{PS-1}, we again find $\gamma=-\pi$. Given that $\bH(E)=R\langle\Psi_{1}\rangle$, Proposition \ref{diff}(i) implies that the action of the semilinear operator $\mff^{*}$ is determined by $\mff^{*}(\Psi_{1})=\mff^{*}(i^{*}\Theta_{1})=-\gamma \Psi_{1}=\pi \Psi_{1}$. 

Observe that in this setting, the relation $(z-\pi)\qH=\mathfrak{p}_{\bH_{\d}(E)}$ holds, which coincides perfectly with the standard Hodge--Pink structure for a Carlitz module (cf.\ \cite[p.\ 1290]{Hartl 2011}). Therefore, for a Carlitz module $E$, our $z$-isocrystal with Hodge--Pink structure $(\wHH,\qH,\mathfrak{f}^{*})$ fully recovers the classical framework established in \cite{Hartl 2011}. In particular, the associated Galois representation aligns precisely with the canonical representation derived from the Tate module.
 
\subsection{The case of maximal splitting number $m=r$}
Assume $E$ is a Drinfeld module of rank $r$ possessing the maximal splitting number $m=r$. We then have $\bH(E)=R\langle\Psi_{1}, \Psi_2, \dots, \Psi_r\rangle$. The semilinear operator $\fra^*$ acts on the quotient $R$-module $\HH(E)$ according to the following rules:
\begin{align*}
\fra^*(\Psi_i) &= \Psi_{i+1}, \quad \text{for all } i = 0, \dots, r-1, \\
\fra^*(\Psi_r) &= \phi(\lam_{r-1})\Psi_r + \dots + \phi(\lam_1)\Psi_1 + \gamma \Psi_1.
\end{align*}

\begin{lemma}\label{theta-lemma} 
Let $E/R$ be a Drinfeld module of rank $r$ with splitting number $r$. If $\Theta_{r}$ is the canonical character in $\bX_{r}(A)$, then 
$$ \Theta_{r} = \Psi_{1} \circ \mf g. $$
\end{lemma}

\begin{proof}
Recall the following commutative diagram:
$$
\xymatrix{ 
 N^{r+1} \ar[r]^-{\phi\circ \iota} \ar[d]_-{P(\mfrak{f})} & J^{r}E \ar[r] \ar@{.>}[dl]_{\mf g} \ar[d]_-{P(\mfrak{\phi})} & \overline{E} \ar[d]_-{P(\overline{F})} \\
 N^1\ar[d]_{\Psi_{1}} \ar[r]^-{\phi \circ \iota} & E \ar[r]^-h & \overline{E} \\
\hG & & 
}
$$
From this construction, it follows that $\Psi_{1}\circ \mf g\in \bX_{r}(E)$. Because $E$ has splitting number $r$, the module $\bX_{r}(E)$ is generated by $\Theta_r$, so $\bX_{r}(E)=R\langle\Theta_{r}\rangle$. Consequently, $\Psi_{1}\circ \mf g = b \cdot \Theta_{r}$ for some $b\in R$. 

By the commutativity of the diagram, we observe: 
\begin{equation}\label{e1}
\Psi_{1}\circ \mf g\circ \phi\circ \iota = \Psi_{1}\circ P(\mff) = \Psi_{r+1} + a_{r-1}\Psi_{r} + \dots + a_{0}\Psi_{1}.
\end{equation}
Conversely, utilizing \eqref{Thetapull}, we compute:
\begin{align}\label{e2}
b \cdot \Theta_{r}\circ \phi\circ \iota &= \iota^{*}\phi^{*}(b \cdot \Theta_{r}) \nonumber\\
&= \phi(b) \cdot (\mff^{*}\iota^{*}\Theta_r + \gamma \Psi_{1}) \nonumber\\
&= \phi(b)(\Psi_{r+1} - \phi(\lambda_{r-1})\Psi_{r} + \dots - \phi(\lambda_1)\Psi_2 + \gamma \Psi_{1}).
\end{align}
Comparing the coefficients of $\Psi_{r+1}$ in equations \eqref{e1} and \eqref{e2}, it is evident that $b=1$, which yields the desired result.
\end{proof}

\begin{corollary}\label{comm-frob} 
Let $E$ be a Drinfeld module satisfying the hypotheses of Lemma \ref{theta-lemma}. Then 
$$ \lambda_{i} = -a_{i} \quad \text{for } 1 \leq i \leq r-1, \quad \text{and} \quad \gamma = a_{0}. $$
\end{corollary}

\begin{proof}
By Lemma \ref{theta-lemma}, we established that $\Theta_{r}=\Psi_{1}\circ \mf g$. Right-composing both sides with $\phi\circ \iota$ and invoking the commutativity of diagram \ref{frob-diag}, we obtain the identity:
$$ \Psi_{r+1} - \phi(\lambda_{r-1})\Psi_{r} - \dots - \phi(\lambda_1)\Psi_{2} + \gamma \Psi_{1} = \Psi_{r+1} + a_{r-1}\Psi_{r} + \dots + a_0 \Psi_{1}. $$
Comparing the coefficients of the basis elements on both sides yields 
$$ \lambda_{i} = -a_{i} \quad \text{for } 1\leq i\leq r-1 ~
\text{(since $\phi$ on $A$ is identity), } \text{and } \gamma = a_{0}. $$
\end{proof}

\begin{theorem} \label{Weil-poly} 
Let $E$ be a Drinfeld module satisfying the hypotheses of Lemma \ref{theta-lemma}. Then the characteristic polynomial of the linear operator $\mff^{*}$ acting on $\bH_{\d}(E)$ is given by 
$$ P(X) = X^r + a_{r-1}X^{r-1} + \dots + a_{0}. $$
\end{theorem}

\begin{proof} 
Recall from \cite[Section 9]{BS b} that the matrix representation of $\mff^{*}$ with respect to the ordered basis $\{\Psi_{1}, \Psi_{2}, \dots, \Psi_r \}$ takes the form:
\[
[\mathfrak{f}^{*}]=
\begin{bmatrix}
    0 & \dots & 0 & -\gamma  \\
    1 & \dots & 0 & \phi(\lambda_1) \\
    \vdots & \ddots & \vdots & \vdots \\
    0 & \dots & 1 & \phi(\lambda_{r-1})
\end{bmatrix}. 
\]
Therefore, the characteristic polynomial of $\mff^{*}$ is easily computed to be 
$$ X^r - \phi(\lambda_{r-1})X^{r-1} - \dots - \phi(\lambda_1)X + \gamma. $$ 
Applying the identities established in Corollary \ref{comm-frob}, this polynomial matches perfectly with 
$$ P(X) = X^r + a_{r-1}X^{r-1} + \dots + a_{0}, $$
which completes the proof.
\end{proof}

 %\subsection{Non-CL Drinfeld modules of rank $2$:}
 %When $E$ is a non-CL Drinfeld module of rank $2$. Then the splitting number of $E$ is $2$ and by Theorem $1.1$ in \cite{PS-3} we have $\bH(E)_K$ is isomorphic to $\Hdr(E)_K$. Therefore we get a Frobenius operator on the de Rham cohomology coming the delta isocrystal. The existence of Frobenius operator on the de Rham cohomology was known for the Drinfeld module over  a finite field (cf.\cite{Angles}, \cite{Gekeler_F}). However, our comparison theorem allow us to have such an operator for global function fields as well.
\footnotesize{
\bibliographystyle{amsalpha}
%\begin{thebibliography}{A}

\end{document}